\documentclass[reqno]{amsart}
\usepackage{amssymb,stmaryrd}
\usepackage{amsfonts}
\usepackage{graphicx}
\usepackage{amsmath}
\usepackage{color}
\usepackage{physics}

\numberwithin{equation}{section}
\newtheorem{theorem}{Theorem}[section]
\newtheorem{lemma}[theorem]{Lemma}
\newtheorem{proposition}[theorem]{Proposition}
\newtheorem{corollary}[theorem]{Corollary}
\theoremstyle{definition}

\newtheorem{remark}[theorem]{Remark}

\newcommand{\R}{{\mathbb{R}}}
\newcommand{\C}{{\mathbb{C}}}
\newcommand{\Z}{{\mathbb{Z}}}
\newcommand{\N}{{\mathbb{N}}}
\newcommand{\<}{{\langle}}
\renewcommand{\>}{{\rangle}}
\newcommand{\tens}{\otimes}
\newcommand{\extd}{{\rm d}}

\newcommand{\eps}{\epsilon}
\newcommand{\id}{{\rm id}}

\allowdisplaybreaks[1]

\begin{document}

\title[Quantum geometry of the discrete interval and $q$-deformation]{Quantum Riemannian geometry of the discrete interval and $q$-deformation}


\author{J. N. Argota-Quiroz and S. Majid}\thanks{{\it Authors to whom correspondence should be addressed:} j.n.argotaquiroz@qmul.ac.uk and s.majid@qmul.ac.uk}
\date{Revised  24 February 2023}
\maketitle

{\ }\vspace{-1.5cm}
\begin{center}\small
School of Mathematical Sciences,\\
 Queen Mary University of London,\\
327 Mile End Rd, London E1 4NS, UK
\end{center}
\bigskip

\begin{quote}\small {\rm ABSTRACT}  We solve for quantum Riemannian geometries on the finite lattice interval  $\bullet-\bullet-\cdots-\bullet$ with $n$ nodes (the Dynkin graph of type $A_n$) and find that they are necessarily $q$-deformed with $q=e^{\imath\pi\over n+1}$. This comes out of the intrinsic geometry and not by assuming any quantum group in the picture.  Specifically, we discover a novel `boundary effect'  whereby, in order to admit a quantum-Levi Civita connection, the `metric weight' at any edge is forced to be greater pointing towards the bulk compared to towards the boundary, with ratio given by $(i+1)_q/(i)_q$ at node $i$, where $(i)_q$ is a $q$-integer.  The Christoffel symbols are also q-deformed. The limit $q\to 1$ likewise forces the quantum Riemannian geometry of the natural numbers $\N$ to have rational metric multiples $(i+1)/i$ in the direction of increasing $i$. In both cases, there is a unique Ricci-scalar flat metric up to normalisation. Elements of quantum field theory and quantum gravity are exhibited for $n=3$ and for the continuum limit of the geometry of $\N$. The Laplacian for the scalar-flat metric becomes the Airy equation operator ${1\over x}{\extd^2\over\extd x^2}$ in so far as a limit exists. Scaling this metric by a conformal factor $e^{\psi(i)}$ gives a limiting Ricci scalar curvature proportional to ${e^{-\psi}\over x}{\extd^2  \psi\over\extd x^2}$. \end{quote}

\section{Introduction}

The idea that spacetime coordinates are better modelled as a noncommutative algebra due to quantum gravity effects has gained traction in recent years as the `quantum spacetime hypothesis'. While speculated on since the early days of quantum mechanics\cite{Sny}, the proper study of this idea only became possible with the arrival of the mathematics of noncommutative geometry and related quantum group symmetries in the 1980s and 1990s. Models connecting these with the Planck scale (actually at the deformed phase space level) appeared in \cite{Ma:pla}, while deformed Minkowski space itself with quantum Poincar\'e group symmetry appeared in \cite{MaRue}, for a quantum group which had been proposed in \cite{Luk}. Other proposals from other contexts included \cite{Hoo,DFR}. These various  models were flat but over the last 30 years there has emerged a  constructive formalism of quantum Riemannian geometry (QRG), see \cite{BegMa} and references therein, which now allows the systematic construction of curved examples. This approach starts with an algebra $\Omega$ of differential forms over the coordinate algebra $A$, formulates a metric as $g\in \Omega^1\tens_A\Omega^1$ and a quantum Levi-Civita connection $\nabla:\Omega^1\to \Omega^1\tens_A\Omega^1$ obeying certain properties\cite{BegMa:gra}.

Remarkably, this theory produces nontrivial results even when $A$ is finite dimensional, for example functions on a finite graph\cite{Ma:gra}. This provides a systematic route to geometry on a finite lattice not as an approximation but as an exact quantum geometry within a single formalism that includes such models at one end and classical GR at the other. In the graph case, functions commute amongst themselves but do not commute with differentials. First quantum gravity models on finite graphs appeared in \cite{Ma:sq,ArgMa1} and the present paper is now third in that particular sequence. Quantum gravity on fuzzy spheres is another effectively finite model\cite{LirMa}. The QRG approach to quantum gravity cannot directly be compared with other approaches such as loop  quantum cosmology\cite{Ash}, dynamical triangulations\cite{Loll} and causal sets\cite{Dow} due to different methods, but some aspects could eventually connect up.  Also note that our conception of QRG is very different from Connes' spectral triple approach which encodes the noncommutative geometry as an axiomatically defined `Dirac operator'\cite{Con}, but the two approaches are not incompatible. Finite models in Connes approach were applied to quantum gravity in \cite{Hal} and have also been applied to the standard model of particle physics\cite{ConMar}.

\begin{figure}
\[\includegraphics[scale=0.9]{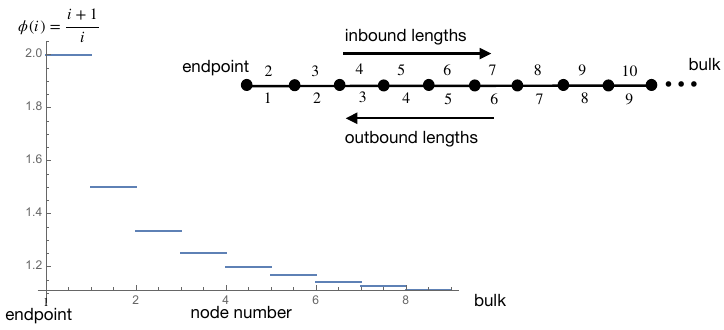}\]
\caption{\label{endpoint} The direction coefficient $\phi(i)={i+1\over i}$ at node $i$ on the half-line $\N$. Metrics that  admit a QRG have an arbitrary real number at each edge but in the ratio shown for the inbound direction / outbound direction. Eg at the first link the inbound length is twice the outbound, at the second the ratio is 3:2, etc. The ratio tends rapidly to 1 as we enter the bulk showing that this is an effect due to the endpoint boundary.}
\end{figure}

Specifically, in this paper we explore the quantum Riemannian geometry of the finite line graph $\bullet-\bullet-\cdots-\bullet$ with $n$ nodes (the Dynkin graph of type $A_n$) as well as the half-line with nodes the natural numbers $\N$. This turns out to be an order of magnitude harder than the case a closed $n$-gon $\Z_n$ or the integers $\Z$ solved in \cite{ArgMa1,Ma:haw} respectively. The complication comes from the boundary, i.e. the endpoints, and we discover a remarkable and unexpected new effect. Namely, it has been pointed out\cite{Ma:gra} that there is nothing in the formalism that says that the length of an edge has to be the direction-independent. A graph metric in QRG is just a nonzero real number assigned to every {\em arrow} with one arrow in each direction for every `link' or graph-edge. To keep things simple and to conform to physical intuition, one usually insists on the metric being edge-symmetric so that these two directions have the same value. The more general case can certainly be considered e.g. the polygon case with asymmetric metrics was recently solved in \cite{Sit} as a generalisation of \cite{ArgMa1}, but there is no particular reason to do so. Our new and rather surprising result is that for the $A_n$ graph with $n>2$ and for $\N$ {\em there is no edge-symmetric QRG}. We are forced to introduce a `direction coefficient' $\phi$ on edges to measure the ratio of the inbound arrow (towards the bulk) metric value compared to the outbound arrow one and find for $\N$ that these have to be specific rational numbers as shown in Figure~\ref{endpoint} in order to admit a quantum-Levi Civita connection. These ratios decay rapidly from 2 at the endpoint down to 1 in the bulk. As long as we keep to these ratios, we are free to vary the actual metric coefficients  or `square-lengths' as we please, so the moduli of QRGs is the same as classically -- a single `square-length' on every link -- but the new effect is that if we consider this as the outbound one then the inbound one is a multiple $\phi$ of it, namely twice at the first link at the left end, 3/2 at the link which is one in from the end, etc.  The metric coefficients have units of length squared as explained in \cite{Ma:sq,BegMa}.

The situation for $A_n$ is similar and we again find a canonical choice of quantum-Levi Civita connection provided the $\phi_i$ at edges $i=0,\cdots,n-1$ are now given by $q$-integers
\[ \phi_i= {(i+1)_q\over (i)_q},\quad q=e^{\imath\pi\over n+1},\quad (i)_q={q^i-q^{-i}\over q-q^{-1}}\]
deforming the canonical QRG for $\N$.  This is the second and equally unexpected discovery of our analysis, that a finite-lattice interval is intrinsically $q$-deformed in its quantum Riemannian geometry, without a quantum group in sight. The fact that $A_n$ is also the Dynkin graph for $SU_{n+1}$ suggests that there could be a role for $u_q(su_{n+1})$ at the specified root of unity, possibly as some kind of diffeomorphism quantum group, but this remains to be established.

An outline of the paper is as follows. We start with a recap of the formalism in Section~\ref{secpre}, including and choice of differential forms on the graph, in Section~\ref{secext}. In fact, the canonical choice of $\Omega$ for us will be $\Omega_{min}$ defined on any graph\cite{BegMa}.  In our case, it is a certain quotient of the preprojective algebra of type $A_n$, an algebra itself of considerable interest in representation theory.  Moreover, we will find that 3-forms and above vanish, so the geometry is in that limited sense 2-dimensional. We take the same form of calculus for $\N$ also, but now with $\Omega^1$ infinite-dimensional.  The analysis of the QRG is obtained in Section~\ref{secn}, building on a preliminary exploration of $A_2$--$A_5$ by hand (and by Mathematica) in Section~\ref{secsmall}. The actual moduli of QLCs is rather rich and includes a parameter $s$ which to be $*$-preserving needs to be a phase $|s|=1$. The allowed metrics also admit a minus sign $\eps$ if we keep track of which coefficients are positive. Then, in Section~\ref{secn}, we extrapolate from this to general formulae for $A_n$ and $\N$ with a uniform solution for the QRG with the freely chosen metric weights $\{h_i\}$, sign parameter $\eps$ in the metric and modulus 1 parameter $s$ in the quantum Levi-Civita connection.  The physical case of all metric coefficients positive requires $\eps=1$ and if we want the Christoffel symbols to also be real then we are forced to $s=\pm1$ (they also simplify vastly on this case). We adopt these canonical forms of the QRGs for the rest of the paper, summarised in Proposition~\ref{canN} for $\N$ and in Corollary~\ref{canAn} for $A_n$.

After finding this canonical form for the QRGs, we then study the induced Laplacian and aspects of scalar field theory on them in Section~\ref{seclap}. The effect of the direction dependence $\phi_i$ on $\N$ translates to a derivative term correction to the Laplacian which alternates with a $(-1)^i$ factor, preventing a straightforward continuum limit. However, this is suppressed as $1/i$ so that as the lattice spacing tends to zero, this complication is pushed to the boundary at 0. A secondary effect of the $\phi_i$ factor is that the overall geometric factor $\beta^{-1}$ in front of the  Laplacian has a correction compared to the same choices of $h_i$ on $\Z$. We analyse this for the case of constant $h_i$ as something like a $1\over x^2$ force towards the origin. Some partitions functions for scalar field theory on the $A_3$ graph are also computed as proof of concept, with respect to a measure for integration on the $A_3$ graph.

In Section~\ref{secqg}, we study the Riemann curvature of our QRGs on $\N$ and $A_n$ and some aspects of quantum gravity. In both cases,  we find a unique set of metric coefficients $\{h_i^{flat}\}$ such that the scalar curvature vanishes (one can similarly solve for any prescribed curvature). For this background, the Laplacian to leading order (again ignoring the suppressed $(-1)^i$ term) is given by the Airy operator ${1\over x}{\extd^2\over\extd x^2}$. We also look at metrics modifying the flat one, $h_i=h_i^{flat}e^{\psi_i}$ and find in the continuum limit that the leading order scalar curvature is ${e^{-\psi}\over x}{\extd^2\psi\over\extd x^2}$. For a natural choice of measure of integration given by the metric itself, the resulting Einstein-Hilbert action in the continuum limit is topological. For $A_3$, we also look at quantum gravity defined by the discrete Einstein-Hilbert action. The paper concludes with some final remarks about further work.

\section{Recap of quantum Riemannian geometry on graphs}\label{secpre}

It is quite important that our geometric constructions are not ad-hoc but part of a general framework which applies
to most unital algebras and is then restricted to the algebra of functions on the vertices of a graph as explained in \cite{Ma:gra}.

\subsection{Quantum Riemannian geometry} \label{secoutline} We will not need the full generality of the theory and give only the bare bones at this general level, for orientation.  Details are in \cite{BegMa}.

We work with a unital possibly noncommutative algebra $A$ viewed as a `coordinate algebra'. We replace the notion of differential structure on a space by specifying a bimodule $\Omega^1$ of differential forms over $A$. A bimodule means we can multiply a `1-form' $\omega\in\Omega^1$ by `functions' $a,b\in A$ either from the left or the right and the two should associate according to
\begin{equation}\label{bimod} (a\omega)b=a(\omega b).\end{equation}
We also need $\extd:A\to \Omega^1$ an `exterior derivative' obeying reasonable axioms, the most important of which is the Leibniz rule
\begin{equation}\label{leib} \extd(ab)=(\extd a)b+ a(\extd b)\end{equation}
for all $a,b\in A$. We usually require $\Omega^1$ to extend to forms of higher degree to give a graded algebra  $\Omega=\oplus\Omega^i$ (where associativity extends the bimodule identity (\ref{bimod}) to higher degree). We also require $\extd$ to extend to $\extd:\Omega^i\to \Omega^{i+1}$ obeying a graded-Leibniz rule with respect to the graded product $\wedge$ and $\extd^2=0$. This much structure is common to most forms of noncommutative geometry, including \cite{Con} albeit there it is not a starting point. In our constructive approach, this `differential structure' is the first choice we have to make in model building once we fixed the algebra $A$. We require that $\Omega$ is then generated by $A,\extd A$ as it would be classically. A first order calculus is inner if there is $\theta\in\Omega^1$ such that $\extd=[\theta, ]$ and similarly with graded commutator for the exterior algebra to be inner.

Next, on an algebra with differential structure, we define a metric as an element $g\in \Omega^1\tens_A\Omega^1$ which is invertible in the sense of a map $(\ ,\ ):\Omega^1\tens_A\Omega^1\to A$ that commutes with the product by $A$ from the left or right and inverts $g$ in the sense
\begin{equation}\label{metricinv}((\omega,\ )\tens_A\id)g=\omega=(\id\tens_A(\ ,\omega))g\end{equation}
 for all 1-forms $\omega$. This requires the metric to be central. In the general theory, one can require quantum symmetry in the form $\wedge(g)=0$, where we consider the wedge product on 1-forms as a map $\wedge:\Omega^1\tens_A\Omega^1\to A$ and apply this to $g$. In practice, we might omit quantum symmetry or impose some variant according to context. 

Finally, we need the notion of a connection. A left connection on $\Omega^1$ is a linear map $\nabla :\Omega^1\to \Omega^1\tens_A\Omega^1$ obeying a left-Leibniz rule
\begin{equation}\label{connleib}\nabla(a\omega)=\extd a\tens_A \omega+ a\nabla \omega\end{equation}
for all $a\in A, \omega\in \Omega^1$. This might seem mysterious but if we think of a map $X:\Omega^1\to A$ that commutes with the right action by $A$ as a `vector field' then we can evaluate $\nabla$ as a covariant derivative $\nabla_X=(X\tens_A\id)\nabla:\Omega^1\to \Omega^1$, which classically would be a usual covariant derivative on $\Omega^1$. There is a similar notion for a connection on a general `vector bundle' expressed algebraically, but we only need the $\Omega^1$ case. Moreover, when we have both left and right actions of $A$ forming a bimodule as we do here, we say that a left connection is a {\em bimodule connection}\cite{DVM,BegMa} if there also exists a bimodule map $\sigma$ such that
\begin{equation}\label{sigma} \sigma:\Omega^1\tens_A\Omega^1\to \Omega^1\tens_A\Omega^1,\quad \nabla(\omega a)=(\nabla\omega)a+\sigma(\omega\tens_A\extd a)\end{equation}
for all $a\in A, \omega\in \Omega^1$.  The map $\sigma$, if it exists, is unique, hence this is not additional data but a property that some connections have.  The key thing is that bimodule connections extend automatically to tensor products  as
\begin{equation} \nabla(\omega\tens_A\eta)=\nabla\omega\tens_A\eta+(\sigma(\omega\tens_A(\ ))\tens_A\id)\nabla\eta\end{equation} for all $\omega,\eta\in \Omega^1$, so that metric compatibility now makes sense as $\nabla g=0$. A connection is called  QLC or `quantum Levi-Civita' if it is  metric compatible and the torsion also vanishes, which in our language amounts to $\wedge\nabla=\extd$ as equality of maps $\Omega^1\to \Omega^2$.
We also have a Riemannian curvature for any connection,
\begin{equation}\label{curv} R_\nabla=(\extd\tens_A\id-\id\wedge\nabla)\nabla:\Omega^1\to \Omega^2\tens_A\Omega^1,\end{equation} where classically one would interior product the first factor against a pair of vector fields to get an operator on 1-forms. Ricci requires more data and the current state of the art (but probably not the only way) is to introduce a lifting bimodule map $i:\Omega^2\to\Omega^1\tens_A\Omega^1$. Applying this to the left output of $R_\nabla$, we are then free to `contract' by using the metric and inverse metric to define ${\rm Ricci}\in \Omega^1\tens_A\Omega^1$ \cite{BegMa}. The associated Ricci scalar and the geometric quantum Laplacian are
\begin{equation}\label{LapA} S=(\ ,\ ){\rm Ricci}\in A,\quad \Delta=(\ ,\ )\nabla\extd: A\to A\end{equation}
defined again along lines that generalise these classical concepts to any algebra with differential structure, metric and connection.

Finally, and critical for physics, are unitarity or `reality' properties. We work over $\C$ but assume that $A$ is a $*$-algebra (real functions, classically, would  be the self-adjoint elements). We require this to extend to $\Omega$ as a graded-anti-involution (so reversing order with an extra sign when odd degree differential forms are involved) and to commute with $\extd$. `Reality' of the metric and of the connection in the sense of being $*$-preserving are imposed as \cite{BegMa},
\begin{equation}\label{realgnab} g^\dagger=g,\quad \nabla\circ *= \sigma\circ\dagger\circ \nabla;\quad (\omega \tens_A\eta)^\dagger=\eta^*\tens_A \omega^*,\end{equation}  where $\dagger$ is the natural $*$-operation on $\Omega^1\tens_A\Omega^1$. These `reality' conditions in a self-adjoint basis (if one exists) and in the classical case would ensure that the metric and connection coefficients  are real.

In practical terms, if the exterior algebra is inner then a connection has the form\cite{Ma:gra}\cite[Prop. 8.11]{BegMa} $\nabla=\theta\tens(\ )-\sigma((\ ) \tens\theta)+\alpha$ for a free choice of bimodule maps $\sigma$ as above and  $\alpha:\Omega^1
\to \Omega^1\tens_A\Omega^1$.  Then torsion free is equivalent to $\wedge\alpha=0$ and $\wedge\sigma=-\wedge$ while $\nabla g=0$ is equivalent to
\begin{equation}\label{metcon} \theta\tens g + (\id\tens\alpha)g + \sigma_{12}(\id\tens(\alpha - \sigma(\tens\theta)))g = 0. \end{equation}
In the $*$-algebra case, if $\theta^*=-\theta$ then we need $(\dagger\circ\sigma)^2=\id$ and $\sigma\circ\dagger\circ\alpha=\alpha\circ *$ for a $*$-preserving connection.

\subsection{Canonical exterior algebras on a graph} \label{secX}

Let $X$ be a discrete set and $A=\C(X)$ the usual commutative algebra of complex functions on it. It can be shown (basically by considering the action of $\delta$-functions) that for such an algebra, the possible differential structures $(\Omega^1,\extd)$ are in 1-1 correspondence with directed graphs with $X$ as the set of vertices, c.f. \cite{Con, Ma:gra,BegMa}. A directed graph just means to draw at most one arrow between some of the vertices, and no self-arrows are allowed. In fact, for the calculus to admit a quantum metric the graph needs to be bidirected, i.e. whenever there is an arrow $x\to y$ there is also an arrow $y\to x$; in other words, our data will just be an undirected graph where $x-y$ means an arrow in both directions. The reason this graph language is useful is that $\Omega^1$ has a basis $\{\omega_{x\to y}\}$ over $\C$ exactly labelled by the arrows of the graph. We then define the  bimodule structure and differential
\[ f.\omega_{x\to y}=f(x)\omega_{x\to y},\quad \omega_{x\to y}.f=\omega_{x\to y}f(y),\quad \extd f=\sum_{x\to y}(f(y)-f(x))\omega_{x\to y}\]
and in the bidirected case a quantum metric has the form \cite{Ma:gra}
\[ g=\sum_{x\to y}g_{x\to y}\omega_{x\to y}\tens_A\omega_{y\to x}\]
with weights $g_{x\to y}\in \R\setminus\{0\}$ for every arrow.  The calculus over $\C$ is compatible with complex conjugation on functions $f^*(x)=\overline{f(x)}$ and $\omega_{x\to y}^*=-\omega_{y\to x}$, from which we see that `reality' of the metric in (\ref{realgnab}) indeed amounts to real metric weights. It is not required mathematically, but reasonable from the point of view of the physical interpretation, to prefer  the {\em edge symmetric} case where $g_{x\to y}=g_{y\to x}$ is independent of the direction, as a variant of `quantum symmetry' in this  context\cite{Ma:haw,BegMa}. First order calculi on sets are always inner with $\theta=\sum_{x\to y}\omega_{x\to y}$, the sum of all arrows.

Finding a QLC for a metric depends on $\Omega^2$, and here there are four canonical choices for $\Omega$ in the sense that they are defined for any graph. They are all quotients of the path algebra which in degree $d$ consists of the $d$-step paths $\omega_{x_0\to x_1}\tens \cdots\tens\omega_{x_{d-1}\to x_d}\in \Omega^1\tens_A\cdots\tens_A\Omega^1$ (this is the tensor algebra of $\Omega^1$ over $A$). We quotient this by the quadratic relations\cite[Prop.~1.40]{BegMa}
\[ \sum_{y: p\to y\to q} \omega_{p\to y}\wedge\omega_{y\to q} =0\]
for all fixed  $p,q$ that obey one of four conditions. This leads to the four exterior algebras forming a diamond:
\[   \begin{array}{rcl} &\Omega_{max}& \\  \swarrow& \quad & \searrow\\ \Omega_{med}& &\Omega_{med'} \\ \searrow  & & \swarrow\\ & \Omega_{min}& \end{array}\]
where the conditions are all $p,q$ such that
\begin{align*} \Omega_{min}&:\quad {\rm all}\  p,q \\
 \Omega_{med}&: \quad p\ne q     \\
 \Omega_{med'}&:\quad    p\to \kern-10pt |\quad q   \\
 \Omega_{max}&: \quad p\ne q,\quad p\to \kern-10pt |\quad q.
\end{align*}
Three of these were explicitly discussed in \cite{BegMa, Ma:boo} while $\Omega_{med'}$ was
recently used in \cite{Gho}. The exterior derivative $\extd$ is given in \cite{BegMa} with only the first two necessarily inner at the level of the exterior algebra.

\subsection{Exterior algebras on $A_n$ and preprojective algebras}\label{secext}

The preprojective algebra of a graph is a quotient of the path algebra of the graph viewed as bidirected (each edge is viewed as a pair of arrows, one in each direction). For the Dynkin graph of type $A_n$ with nodes numbered in order $1,2,\cdots, n$, we denote the edges
\[ a_i=\omega_{i\to i+1},\quad a_i'=\omega_{i+1\to i}=-a_i^*,\quad i=1,\cdots,n-1.\]
In the maths literature, the notation $a_i^*$ is used for what we denote $a_i'$; the two differ by a sign which just amounts to a different normalisation but is needed for our exterior algebras to become $*$-exterior algebras when working over $\C$. We also denote by $\delta_i$ the Kronecker $\delta$-functions at the nodes. The path algebra then has the relations that all products of these generators are zero except
 \[ \delta_i^2=\delta_i,\quad \delta_i a_i=a_i=a_i\delta_{i+1},\quad \delta_{i+1}a_i'=a_i'=a_i'\delta_i,\quad a_ia_{i+1},\quad  a_{i+1}'a_i',\quad a_ia_i',\quad a_i'a_i\]
 The dimension of the path algebra in degree 0 is $n$ with basis $\delta_i$. In degree 1 it is $2(n-1)$-dimensional with basis $a_i,a_i'$ and in degree 2 it is $2(2n-3)$-dimensional with basis $a_ia_{i+1},a_{i+1}'a_i'$ for $i=1,\cdots,n-2$ and $a_ia_i', a_i'a_i $ for $i=1,\cdots,n-1$.

\begin{proposition} For a Dynkin graph of type $A_n$, $\Omega_{max}=\Omega_{med}$ is a quotient of the path algebra  by the relations
\begin{equation}\label{maxrel} a_ia_{i+1}=0,\quad  a_{i+1}'a_i' =0,\quad i=1,\cdots,n-2\end{equation}
and $\Omega_{min}=\Omega_{med'}$ is the further quotient by the relations
\begin{equation}\label{pirel} a_1a_1'=0,\quad a_{n-1}'a_{n-1}=0,\quad a_{i+1}a_{i+1}'+ a_i'a_i=0,\quad i=1,\cdots,n-2.\end{equation}
The latter case is inner with $\extd =[\theta,\ \}$, where $\theta=\sum_i a_i+a_i'$, and $\Omega^2_{min}$ is $n-2$-dimensional while  $\Omega_{min}^i=0$ for $i\ge 3$.
\end{proposition}
\begin{proof}The dimensions up to $\Omega^2$ are clear from the stated bases and quadratic nature of the relations. In degree 3 we consider all 3-step paths and their image in $\Omega^3_{min}$. Since any 2-steps in the same direction vanish by  (\ref{maxrel}), the only possible images in the quotient are for zig-zag paths such as $a_{i+1}'a_{i+1}a_{i+1}'=-a'_{i+1}a_i'a_i=0$ using (\ref{pirel}) and then (\ref{maxrel}). Similarly for zig-zag the other way, $a_ia_i'a_i=-a_ia_{i+1}a_{i+1}'=0$.  \end{proof}

The preprojective $\Pi_n$ has just the (\ref{pirel})  relations and  dimensions
\[ n,\ 2(n-1),\ 3(n-2),\ ... ,\ (n-1)2,\ n.\]
 We see that
$\Omega_{min}$ is a quotient of this by (\ref{maxrel}). Also, later, we will need a bimodule `lifting' map $i:\Omega^2_{min}\to \Omega^1\tens_A\Omega^1$ such that following this by $\wedge$ is the identity. Given the description above, the natural choice is
\begin{equation}\label{imin} i(a_i a_i')=-i( a_{i-1}'a_{i-1})={1\over 2}(a_i\tens a_i'- a_{i-1}'\tens a_{i-1}),\quad i=2,3,\cdots,n-1\end{equation}
where the product denotes wedge product. We take the same form of exterior algebra relations and lift map for the half-line $\N$, just without the upper bound on the indices $i$. Also note from the form of the path algebra that we can only have zero for the bimodule map $\alpha$, i.e. bimodule connections $\nabla$ are determined just from $\sigma$.

\section{Explicit calculations for $A_2,A_3,A_4,A_5$}\label{secsmall}

In this section, we give explicit solutions for small $A_n$. For $n\le 4$, these are manageable by hand and we show the details of the calculation. For $n=5$, we used Mathematica and Python (independently) and just list the final result.

\subsection{$A_2$ geometry}

The result is known from \cite{Ma:sq} by a different method, but here provides a warm up for the larger cases. We work over the directed graph $G(V,E)$ with vertices $V = \{1,2\}$ and directed edges or `arrows' $E = \{a_1,a_1'\}$ as in Figure~\ref{fig:A2}. The products distinct from zero in $\Omega^1\tens_A\Omega^1$ are $a_1\tens a_1', a_1'\tens a_1$. The exterior algebra $\Omega^2_{max}$ is 2-dimensional with basis $a_1\wedge a_1'$ and $a_1'\wedge a_1$. We work with $\Omega_{min}$ where these are set to zero.

\begin{figure}
    \centering
      \includegraphics[scale=0.6]{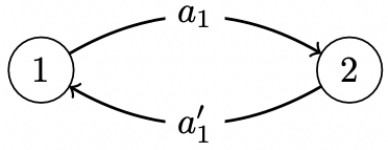}
 \caption{\label{fig:A2} $A_2$ Graph}
\end{figure}

Using the graded commutator for the exterior derivatives given that the calculus is inner with $\theta = a_1 + a_1'$,
\begin{align*}
    \extd a = [ \theta , a_1 \} =  a_1'\tens a_1 + a_1 \tens a_1' = \extd a_1',
\end{align*}

We necessarily take $\alpha=0$ and the most general form of $\sigma$ is 
\begin{align*}
    \sigma(a_1\tens a_1') = \tau_1 a_1\tens a_1', \quad \sigma(a_1'\tens a_1) = \tau'_1 a_1'\tens a_1
\end{align*}

The general form of the metric is
\[g = f_1 a_1\tens a_1' + f'_1 a_1'\tens a_1\]
where $f_1, f'_1$ are in the field, and real if we work over $\C$ and impose the reality condition for the metric.

The general form of the connection given that the calculus is inner is
\begin{align*}
    \nabla a_1 = a_1'\tens a_1 - \tau_1 a_1\tens a_1', \quad  \nabla a_1' = a_1 \tens a_1' - \tau'_1 a_1'\tens a_1
\end{align*}
As we are working in $\Omega_{min}$, there is no elements in $\Omega^2$ for this case. Then there are no conditions for torsion freeness.

The metric compatibility conditions are
\begin{align*}
    a_1 \tens a_1'\tens a_1 : f'_1 - f_1\tau_1\tau'_1 = 0  \\
    a_1'\tens a_1 \tens a_1': f_1 - f'_1\tau_1\tau'_1 = 0
\end{align*}

These conditions imply that there is a sign $\eps=\tau_1\tau'_1=\pm1$ with $f'_1 = \eps f_1 $. The $*$-preserving conditions
\[\abs{\tau_1} = 1\]
with $\tau'_1=\eps\tau_1^{-1}$. We see that there is one sign and one overall normalisation in the metric
\[ g = h_1 (a_1\tens a_1' +  \eps a_1'\tens a_1)\]
which allows a QLC with one parameter $\tau_1=s$ in characteristic zero
\begin{align*}
    \nabla a_1 = a_1'\tens a_1 -s a_1\tens a_1', \quad  \nabla a_1' = a_1 \tens a_1' -\eps s^{-1} a_1'\tens a_1.
\end{align*}
and the further condition that $h_1$ is real and $|s|=1$ for the reality property of the metric and for the connection to be $*$-preserving in the case over $\C$.  All the connections are flat since $\Omega^2=0$.

\subsection{$A_3$ geometry}

We work over the directed graph $G(V,E)$ with vertices $V = \{1,2,3\}$ and directed edges $E = \{a_1,a_1',a_2,a_2'\}$  as in Figure~\ref{fig:A3}. The products in $\Omega^1\tens_A\Omega^1$ different from zero are those where the head of the first arrow connects to the tail of the second arrow, giving the six non-zero elements $a_1\tens a_1', a_1\tens a_2, a_1'\tens a_1, a_2\tens a_2', a_2'\tens a_1', a_2'\tens a_2$.

\begin{figure}
       \centering
   \includegraphics[scale=0.6]{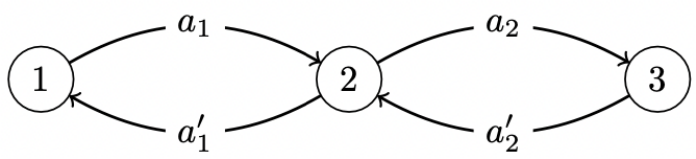}
  \caption{ \label{fig:A3} $A_3$ Graph}
\end{figure}

The exterior algebra $\Omega_{max}$ for the maximal prolongation has the relations
\begin{align*}
    a_1 \wedge a_2 = a_2' \wedge a_1' = 0
\end{align*}
and we work with the quotient $\Omega_{min}$ of this where we add the further relations
\begin{align*}
    a_1 \wedge a_1' = a_2' \wedge a_2 = 0,\quad a_1' \wedge a_1 + a_2 \wedge a_2' = 0.
\end{align*}
The dimensions of the vector spaces of $\Omega^i$ are therefore 3:4:1.

The exterior derivative is  given by the graded commutator $\extd = [\theta, \}$ with the inner element $\theta = a_1 + a_1' + a_2 + a_2'$ as
\begin{align*}
    \extd a_1 = a_1'\wedge a_1, \quad
    \extd a_1'= a_1'\wedge a_1,\quad \extd a_2 =-a_1'\wedge a_1, \quad
    \extd a_2'=-a_1'\wedge a_1.
\end{align*}

The metric, as it has to be central, has to have the form
\[g = f_1 a_1\tens a_1' + f'_1 a_1'\tens a_1 + f_2 a_2\tens a_2' + f'_2 a_2'\tens a_2,\]
where $f_1,f'_1,f_2,f'_2$ are in the field, and should be real if we work over $\C$ and impose the reality condition $\dagger \circ g = g$.

Given the calculus is inner, the torsion free connections have the form
\begin{align*}
    \nabla a_1  &=  a_1'\tens a_1 - \tau_1 a_1\tens a_1' - \sigma_1 a_1\tens a_2,\\
    \nabla a_1' &=  a_1\tens a_1' + a_2'\tens a_1' - \tau_1' a_1'\tens a_1 - (\tau_1'+1) a_2\tens a_2',\\
    \nabla a_2  &= a_1\tens a_2 + a_2' \tens a_2 - \tau_2 a_2\tens a_2' - (\tau_2+1) a_1'\tens a_1,\\
    \nabla a_2' &= a_2\tens a_2' - \sigma_2' a_2'\tens a_1' - \tau'_2 a_2'\tens a_2,
\end{align*}
where the map $\alpha=0$ and the braiding map is given by
\begin{align*}
    \sigma(a_1\tens a_1') &= \tau_1 a_1\tens a_1',\\
    \sigma(a_1\tens a_2) &=  \sigma_1 a_1\tens a_2 \\
    \sigma(a_1'\tens a_1) &= \tau_1' a_1'\tens a_1  + (\tau_1'+1) a_2\tens a_2',\\
    \sigma(a_2\tens a_2') &= \tau_2 a_2\tens a_2' + (\tau_2+1) a_1'\tens a_1,\\
    \sigma(a_2'\tens a_1') &=\sigma_2'  a_2'\tens a_1',\\
    \sigma(a_2'\tens a_2) &= \tau'_2 a_2'\tens a_2.
\end{align*}

Metric compatibility (\ref{metcon}) then produces
\begin{align*}
     a_1\tens a_1'\tens a_1&:   -f_1 \tau_1 \tau_1' +  f'_1 = 0,\\
     a_1\tens a_2 \tens a_2' &: -f_1 \sigma_1 (\tau_1' +1) +  f_2 = 0,\\
    a_1'\tens a_1 \tens a_1' &: -f'_1 \tau_1 \tau_1' -  f_2 (\tau_2 +1) \sigma_2' +  f_1= 0,\\
    a_2 \tens a_2'\tens a_2 &:  -f'_1 \sigma_1 (\tau_1' +1) -  f_2 \tau_2 \tau'_2 +  f'_2 = 0,\\
    a_2'\tens a_1'\tens a_1&:   -f'_2 (\tau_2 +1) \sigma_2' +  f'_1 = 0,\\
    a_2'\tens a_2 \tens a_2' &: -f'_2 \tau_2 \tau'_2 +  f_2 = 0,\\
    a_2 \tens a_2'\tens a_1' &: -f'_1 \tau_1 (\tau_1' +1) -  f_2 \tau_2 \sigma_2'= 0,\\
    a_1'\tens a_1 \tens a_2 &:  -f'_1 \sigma_1 \tau_1' -  f_2 (\tau_2 +1) \tau'_2= 0.
\end{align*}

Under these conditions, we have two parameters and one sign in the metric as
\begin{align*}
    g = h_1 (   \phi a_1\tens a_1' + \eps a_1'\tens a_1) + h_2( \frac{1}{\phi} a_2\tens a_2' + \eps   a_2'\tens a_2), \quad \phi =\sqrt{2},
\end{align*}
where the connection is
\begin{align*}
    \tau_1& = s, \quad
    \sigma_1 = \frac{{h_2} {s}}{{h_1} \eps\phi  (\eps\phi  {s}+1)}, \quad
    \tau'_1 = \frac{1}{\eps\phi  s}, \quad \\
    \tau_2 &= -1 + \frac{1}{2 + \eps\phi  s}, \quad
    \sigma'_2 = \frac{ h_1 \eps\phi }{h_2}(s + \eps\phi ),    \quad
    \tau'_2 = -\frac{1}{\eps\phi }\left( 1 + \frac{1}{1 + \eps\phi   s} \right)
\end{align*}
for a free parameter $s$.  Notice that only the combination $\eps\phi$ enters. We do not consider $\phi=-\sqrt{2}$ in the metric as this would be equivalent to  a redefinition of $\eps,h_1,h_2$ by a change of sign. Finally, the *-preserving condition for the connection just requires
\begin{align}
    \abs{s} = 1
\end{align}
with no further constraints on $h_i$ other than to be real.

\subsection{$A_4$ geometry}

We again work with $\Omega_{min}$ which now has vector space dimensions $4:6:2$ with $\Omega^i=0$ for $i\ge 3$. Here, the path algebra is 10-dimensional in degree 2, $\Omega^2_{max}$ adds 4 relations  and then we add further relations for $\Omega^2$,
\begin{align*}
 a_1\wedge a_1'=a_3'\wedge a_3=0,\quad   a_1'\wedge a_1 + a_2\wedge a_2' = 0, \quad a_2'\wedge a_2 + a_3\wedge a_3' = 0.
\end{align*}
The metric, to be central, has to have the form
\[g = f_1 a_1\tens a_1' + f'_1 a_1'\tens a_1+f_2 a_2\tens a_2' + f'_2 a_2'\tens a_2+f_3 a_3\tens a_3' + f'_3 a_3'\tens a_3\]
where $f_1,f'_1,f_2,f'_2,f_3,f'_3$ are in the field, and real for the reality condition $\dagger \circ g = g$ when working over $\C$. We necessarily take $\alpha=0$ and the torsion free connection and bimodule braiding map have to have the form
\begin{align*}
    \nabla a_1  &=  a_1'\tens a_1 - \tau_1 a_1\tens a_1' - \sigma_1 a_1\tens a_2,\\
    \nabla a_1' &=  a_1\tens a_1' + a_2'\tens a_1' - \tau_1' a_1'\tens a_1 -  (\tau_1'+1) a_2\tens a_2',\\
    \nabla a_2  &= a_1\tens a_2 + a_2' \tens a_2 - \tau_2 a_ 2\tens a_2' - (\tau_2+1) a_1'\tens a_1 - \sigma_2 a_2\tens a_3,  \\
    \nabla a_2' &= a_2\tens a_2' + a_3'\tens a_2' - \sigma_2' a_2'\tens a_1' - \tau_2' a_2'\tens a_2 -(\tau_2'+1) a_3\tens a_3' ,\\
    \nabla a_3  &= a_2\tens a_3 + a_3'\tens a_3 - \tau_{3} a_3 \tens a_3' - (\tau_{3}+1) a_2'\tens a_2,\\
    \nabla a_3' &=  a_3\tens a_3' - \sigma_{3}'a_3' \tens a_2' - \tau'_3 a_3'\tens a_3,
\end{align*}
\begin{align*}
    \sigma( a_1 \tens a_1') &= \tau_1 a_1 \tens a_1' ,\\
    \sigma( a_1 \tens a_2 ) &= \sigma_1 a_1 \tens a_2 ,\\
    \sigma( a_1'\tens a_1 ) &= \tau_1' a_1'\tens a_1 + (\tau_1' + 1 ) a_2 \tens a_2' ,\\
    \sigma( a_2 \tens a_2') &= \tau_2 a_2 \tens a_2'+ (\tau_2+1) a_1'\tens a_1 ,\\
    \sigma( a_2 \tens a_3 ) &= \sigma_2 a_2 \tens a_3 ,\\
    \sigma( a_2'\tens a_1') &= \sigma_2' a_2'\tens a_1' ,\\
    \sigma( a_2'\tens a_2 ) &= \tau_2' a_2'\tens a_2 + (\tau_2'+1) a_3 \tens a_3',\\
    \sigma( a_3 \tens a_3') &= \tau_{3} a_3 \tens a_3' + (\tau_{3}+1) a_2'\tens a_2 ,\\
    \sigma( a_3'\tens a_2') &= \sigma_{3}' a_3'\tens a_2' ,\\
    \sigma( a_3'\tens a_3 ) &= \tau'_3 a_3'\tens a_3.
\end{align*}

The metric compatibility conditions are
\begin{align}
    a_1 \tens a_1'\tens a_1&: - f_1 \tau_1 \tau_1' +  f'_1 = 0, \\
    a_1 \tens a_2 \tens a_2'&:- f_1 \sigma_1 (\tau_1'+1) +  f_2 = 0, \\
    a_1'\tens a_1 \tens a_1'&:- f'_1 \tau_1 \tau_1' -  f_2 (\tau_2+1) \sigma_2' +  f_1 = 0, \\
    a_2 \tens a_2'\tens a_2&: - f'_1 \sigma_1 (\tau_1'+1) -  f_2 \tau_2 \tau_2' +  f'_2 = 0, \\
    a_2 \tens a_3 \tens a_3'&:- f_2 \sigma_2 (\tau_2'+1) +  f_3 = 0, \\
    a_2'\tens a_1'\tens a_1&: - f'_2 (\tau_2+1) \sigma_2' +  f'_1 = 0, \\
    a_2'\tens a_2 \tens a_2'&:- f_3 (\tau_3+1) \sigma_3' -  f'_2 \tau_2 \tau_2' +  f_2 = 0, \\
    a_3 \tens a_3'\tens a_3&: - f_3 \tau_3 \tau'_3 -  f'_2 \sigma_2 (\tau_2'+1) +  f'_3 = 0, \\
    a_3'\tens a_2'\tens a_2&: - f'_3 (\tau_3+1) \sigma_3' +  f'_2 = 0, \\
    a_3'\tens a_3 \tens a_3'&:- f'_3 \tau_3 \tau'_3 +  f_3 = 0, \\
    a_2 \tens a_2'\tens a_1'&:- f'_1 \tau_1 (\tau_1'+1) -  f_2 \tau_2 \sigma_2' = 0, \\
    a_1'\tens a_1 \tens a_2&: - f'_1 \sigma_1 \tau_1' -  f_2 (\tau_2+1) \tau_2' = 0, \\
    a_3 \tens a_3'\tens a_2'&:- f_3 \tau_3 \sigma_3' -  f'_2 \tau_2 (\tau_2'+1) = 0, \\
    a_2'\tens a_2 \tens a_3&: - f_3 (\tau_3+1) \tau'_3 -  f'_2 \sigma_2 \tau_2' = 0.
\end{align}

There are four metrics that allow one QLC each depending on a free parameter $s$. Here, the metric is found to be of the form
\[g = h_1( \phi a_1 \tens a_1' +  \eps a_1' \tens a_1)
+ h_2 ( a_2 \tens a_2' + \eps a_2' \tens a_2 )
+ h_3 ( \frac{1}{\phi} a_3 \tens a_3' +  \eps a_3' \tens a_3),\]
where $\eps=\pm1$ and
\[ \phi  = \frac{1 \pm \sqrt{5}}{2}.\]
We can chose either in what follows (so we have four metrics according to $\eps$ and the sign of the $\sqrt{5}$). For any choices of these we now solve for the connection and find for any value of $s$,
\begin{align*}
    \tau_1 &= s, \quad
    \tau_2 = -1 + {1\over \phi + \eps s}, \quad
    \tau_{3} = -1 + {1\over \phi}{ \phi + \eps s  \over (1-\eps)(\phi+\eps s)+\eps } ,\\
    \tau_1' &=  \eps{ -1 + \phi \over s}, \quad
    \tau_{2}'= \eps {(\phi + \eps s)(1 - 1/ \phi)\over  -(\phi + \eps s) + 1}, \quad
    \tau_{3}' =  {\eps\over\phi\tau_3}, \quad
    \sigma_1 = {h_2\over h_1 \phi}{s \over \eps(-1 + \phi)+s}, \\
    \sigma_2 &= {h_2\over h_3\phi}{1\over \tau_2' + 1}, \quad
    \sigma_{2}' = {h_1\over h_2}(\phi + \eps s), \quad
    \sigma_{3}' = {h_2\over h_3}\phi \left( 1 - \eps + {\eps\over \phi + \eps s} \right).
\end{align*}
Thus, for each metric we have  a 1-parameter family of connections with parameter $s$. In the $*$-algebra case, reality of the metric demands $h_i$ real and  $*-$preserving for the connection is  equivalent to
\[\abs{s} = 1\]
with no further constraints on $h_i$. So, there is a still  a 1-parameter moduli of connections for each of our four metrics,  now with $|s|=1$.

\subsection{$A_5$ geometry}

Now  the path algebra has dimension 14 in degree 2 while $\Omega_{min}$  has vector space dimensions $5:8:3$ again with $\Omega^i=0$ for $i\ge 3$. Proceeding similarly, the form of the braiding, connection, metric compatibility conditions and form of the metric are
\begin{align*}
    \sigma (a_1 \tens a_1') &= \tau_1 a_1 \tens a_1',\\
    \sigma (a_1 \tens a_2 ) &= \sigma_1 a_1 \tens a_2,\\
    \sigma (a_1'\tens a_1 ) &= \tau_1' a_1'\tens a_1 + (\tau_1' + 1) a_2 \tens a_2' ,\\
    \sigma (a_2 \tens a_2') &= \tau_2  a_2 \tens a_2' + (\tau_2 + 1) a_1' \tens a_1, \\
    \sigma (a_2 \tens a_3 ) &= \sigma_2 a_2 \tens a_3,\\
    \sigma (a_2'\tens a_2 ) &= \tau_2' a_2' \tens a_2 + (\tau_2' + 1) a_3 \tens a_3', \\
    \sigma (a_2'\tens a_1') &= \sigma_2' a_2' \tens a_1,\\
    \sigma (a_3 \tens a_3') &= \tau_3 a_3 \tens a_3' + (\tau_3 + 1) a_2' \tens a_2, \\
    \sigma (a_3 \tens a_4 ) &= \sigma_3 a_3 \tens a_4, \\
    \sigma (a_3'\tens a_3 ) &= \tau_3' a_3' \tens a_3 + (\tau_3' + 1) a_4 \tens a_4', \\
    \sigma (a_3'\tens a_2') &= \sigma_3' a_3' \tens a_2',\\
    \sigma (a_4 \tens a_4') &= \tau_4 a_4 \tens a_4' + (\tau_4 + 1) a_3' \tens a_3, \\
    \sigma (a_4'\tens a_3') &= \sigma_4' a_4' \tens a_3',\\
    \sigma (a_4'\tens a_4 ) &= \tau'_4 a_4' \tens a_4;
\end{align*}
\begin{align*}
    \nabla a_1 &= a_1' \tens a_1 - \tau_1 a_1 \tens a_1' - \sigma_1 a_1 \tens a_2,  \\
    \nabla a_1' &= a_1 \tens a_1' + a_2' \tens a_1' -\tau_1' a_1'\tens a_1 - (\tau_1' + 1) a_2 \tens a_2',\\
    \nabla a_2 &= a_2' \tens a_2  + a_1 \tens a_2 - \tau_2  a_2 \tens a_2' - (\tau_2 + 1) a_1' \tens a_1  - \sigma_2 a_2 \tens a_3, \\
    \nabla a_2' &= a_2 \tens a_2' + a_3' \tens a_2' - \tau_2' a_2' \tens a_2 - (\tau_2' + 1) a_3 \tens a_3' - \sigma_2' a_2' \tens a_1, \\
    \nabla a_3 &= a_3' \tens a_3  + a_2 \tens a_3 -\tau_3 a_3 \tens a_3' - (\tau_3 + 1) a_2' \tens a_2 -\sigma_3 a_3 \tens a_4, \\
    \nabla a_3' &= a_3 \tens a_3' + a_4' \tens a_3' -\tau_3' a_3' \tens a_3 - (\tau_3' + 1) a_4 \tens a_4' -\sigma_3' a_3' \tens a_2',\\
    \nabla a_4 &= a_4' \tens a_4  + a_3 \tens a_4 -\tau_4 a_4 \tens a_4' - (\tau_4 + 1) a_3' \tens a_3, \\
    \nabla a_4' &= a_4 \tens a_4' - \sigma_4' a_4' \tens a_3' - \tau'_4 a_4' \tens a_4;
\end{align*}
\begin{align*}
    a_1 \tens a_1' \tens a_1 &:     -f_1 \tau_1  \tau_1' + f'_1 = 0,\\
    a_1 \tens a_2 \tens a_2' &:     -f_1 \sigma_1 (\tau_1'+1) + f_2 = 0, \\
    a_1' \tens a_1 \tens a_1' &:    -f'_1 \tau_1  \tau_1' - f_2 (\tau_2 + 1)  \sigma_2' + f_1 = 0, \\
    a_2 \tens a_2' \tens a_2 &:     -f'_1 \sigma_1 (\tau_1' + 1) - f_2  \tau_2  \tau_2' + f'_2 = 0, \\
    a_2 \tens a_3 \tens a_3' &:     -f_2  \sigma_2 (\tau_2' + 1) + f_3 = 0, \\
    a_2' \tens a_1' \tens a_1 &:    -f'_2 (\tau_2 + 1)  \sigma_2' + f'_1 = 0, \\
    a_2' \tens a_2 \tens a_2' &:    -f_3 (1+\tau_3)\sigma_3' - f'_2  \tau_2  \tau_2' + f_2 = 0, \\
    a_3 \tens a_3' \tens a_3 &:     -f_3 \tau_3 \tau_3' - f'_2  \sigma_2 (\tau_2' + 1) + f'_3 = 0, \\
    a_3 \tens a_4 \tens a_4' &:     -f_3 \tau_3 \sigma_{3}' + f_4 = 0, \\
    a_3' \tens a_2' \tens a_2 &:    -f'_3 (1+\tau_3)\sigma_3' + f'_2 = 0, \\
    a_3' \tens a_3 \tens a_3' &:    -f'_3 \tau_3 \tau_3' - f_4 (\tau_4 + 1) \sigma_{4}' + f_3 = 0, \\
    a_4 \tens a_4' \tens a_4 &:     -f'_3 \tau_3 \sigma_{3}' - f_4 \tau_{4} \tau'_4 + f'_4 = 0, \\
    a_4' \tens a_3' \tens a_3 &:    -f'_4 (\tau_4 + 1) \sigma_{4}' + f'_3 = 0, \\
    a_4' \tens a_4 \tens a_4' &:    -f'_4 \tau_{4} \tau'_4 + f_4 = 0, \\
    a_2 \tens a_2' \tens a_1' &:    -f'_1 \tau_1 (\tau_1'+1) - f_2  \tau_2  \sigma_2' = 0, \\
    a_1' \tens a_1 \tens a_2 &:     -f'_1 \sigma_1  \tau_1' - f_2 (\tau_2 + 1)  \tau_2' = 0, \\
    a_3 \tens a_3' \tens a_2' &:    -f_3 \sigma_3' \tau_{3} - f'_2  \tau_2 (\tau_2' + 1) = 0, \\
    a_2' \tens a_2 \tens a_3 &:     -f_3 (\tau_3+1)\tau_3' - f'_2  \sigma_2  \tau_2' = 0, \\
    a_4 \tens a_4' \tens a_3' &:    -f'_3 (\tau_3 + 1) \sigma_{3}' - f_4 \tau_{4} \sigma_{4}' = 0, \\
    a_3' \tens a_3 \tens a_4 &:     -f'_3 \tau_3 (\tau_3' + 1) - f_4 (\tau_4 + 1) \tau'_4 = 0;
\end{align*}
\[ g = h_1 (\phi_1 a_1 \tens a_1' + \epsilon a_1'\tens a_1) + h_2 (\phi_2 a_2 \tens a_2' + \epsilon a_2'\tens a_2) + h_3 (\frac{1}{\phi_2} a_3 \tens a_3' + \epsilon a_3'\tens a_3) + h_4(\frac{1}{\phi_1} a_4 \tens a_4' + \epsilon a_4'\tens a_4), \]
where
\[\phi_1=\sqrt{3},\quad \phi_2={2\over\sqrt 3}\]
and $\eps$ is a sign. In the $*$-algebra case over $\C$ we need $h_i$ real and $|s|=1$ for the reality of the metric and for the connection to be $*$-preserving.

The solutions are then as follows. The parts which looks similar to the $A_4$ case are\begin{align*}
    \tau_1& = s, \quad
    \sigma_1 = \frac{{h_2}\epsilon\phi_2 }{{h_1}   (\epsilon\phi_1  +{1\over s})}, \quad
    \tau_1' = \frac{1}{\epsilon\phi_1  s}, \quad  \tau_2 &= -1 + \frac{\epsilon\phi_2}{\epsilon\phi_1 +   s},\quad
    \sigma_2' = \frac{ h_1  }{h_2\epsilon \phi_2}(\epsilon\phi_1+s )
\end{align*}
and the remainder explicitly are
\[ \tau_3=\frac{-2 \epsilon\phi_1  s+\epsilon\phi_1 +s-2}{2\epsilon \phi_1  (s-2)-4 s+2}, \quad \tau_4=\frac{\epsilon\phi_1  (s-2)-4 s+2}{\epsilon\phi_1  (4-2 s)+6 s-3},\]
\[ \tau'_2=-{\epsilon\phi_1+s\over 2(\epsilon\phi_1 s+1)},\quad  \tau'_3=\frac{\epsilon\phi_1  (s-2)-2 s+1}{\epsilon\phi_1  (s-2)-6 s+3},\quad \tau'_4=\frac{2 (\epsilon\phi_1 -1) s-\epsilon\phi_1 +4}{\epsilon\phi_1  (s-2)-4 s+2},\]
\[ \sigma_2=\frac{3 h_3  (\epsilon\phi_1  s+1)}{2 h_2  (\epsilon\phi_1  (2 s-1)-s+2)}, \quad \sigma_3=\frac{h_4  (\epsilon\phi_1  (2 s-1)-s+2)}{h_3  (\epsilon\phi_1  (4 s-2)-3 s+6)},\]
\[ \sigma_3'=-\frac{2 h_2  (\epsilon\phi_1  (s-2)-2 s+1)}{3 h_3  (\epsilon\phi_1 +s)}, \quad \sigma'_4=\frac{h_3  (2 \epsilon\phi_1  (s-2)-6 s+3)}{h_4  (\epsilon\phi_1  (s-2)-2 s+1)}.\]

\section{Canonical metrics and QLC for $A_n$ and the half-line $\N$}\label{secn}

Here, we solve the system of equations for a quantum Riemannian geometry in general, building on our experience for small $n$.

\subsection{Summary of computer results for $n\le 8$}

\begin{table}\begin{tabular}{|c|c|c|c|c|c}\hline $n$ & ${\rm dim}\Omega^i$ & metrics with QLC & QLC  & $*$-QLC \\
\hline\hline
2 & 2:2:0:0 & $\eps, h_1 $ & $s$   &$ |s|=1$\\
\hline
3 & 3:4:1:0 &  $\eps, h_1,h_2$ & $s$  & $|s|=1$\\ \hline
4&4:6:2:0 & $\eps, \eps', h_1, h_2, h_3$ & $s$  & $|s|=1$ \\ \hline
5&5:8:3:0 & $\eps, h_1, h_2, h_3,h_4$ & $s$  & $|s|=1$ \\ \hline
$n$ & $n:2(n-1):n-2:0 $ & $\eps,\eps', h_1,..., h_{n-1}$ & $s$ &  $|s|=1$ \\ \hline
\end{tabular}
\caption{Summary of vector space dimensions of the exterior algebra and the parameterisation of quantum Riemannian geometries found on $A_n$ for $n\le 5$. \label{table0}}
\end{table}

We summarise the results so far as the first entries in Table~\ref{table0},  
where $h_i$ are real variables for the reality conditions if we work over $\C$, $\eps$ a sign and $\eps'$ is a discrete parameter (not necessarily a binary choice) indicating a discrete moduli for certain numerical `direction coefficients' $\{\phi_i\}$. The results found so far then fit the general  format
\begin{equation}\label{ghphi} g=\sum_{i=1}^{n-1} h_i( \phi_i a_i\tens a_i'+\eps a_i'\tens a_i);\quad \phi_{n-1}={1\over\phi_1},\quad \phi_{n-2}={1\over \phi_2}\end{equation}
etc., as depicted in Figure~\ref{metric}. This means that only the first $\phi_1,\cdots,\phi_{\lfloor {n\over 2}\rfloor}$ have to be specified, the rest are inverse, and that in the even case $\phi_{n\over 2}^2=1$. We also note that
\[ h_i\mapsto -h_i,\quad \phi_i\mapsto -\phi_i,\quad \eps\mapsto -\eps\]
is a symmetry of the metric in the odd case. Hence, without loss of generality, we may assume that $\phi_{n\over 2}=1$ in the even case and in the odd case we do not need to list both a solution for $\{\phi_i\}$ and their negation.  We then solved by computer for all remaining $n\le 8$ and found that all solutions fit this general format with $\{\phi_i\}$ summarised in Table~\ref{table1}. Due to the above symmetry, for $A_2,A_3,A_5$ we do not list a discrete moduli of  $\{\phi_i\}$ as this can be absorbed in a change of sign of the $h_i$ and $\eps$, while in the other cases we list them as separate rows. The solutions involve square roots (as we saw up to $A_5$) or are roots of higher order polynomials. But with some work, we  recognised all entries in a trigonometric form. For example, in the $7^-$ row 
\[\sqrt{2-\sqrt{2}}=2 \cos(\frac{3 \pi }{8}),\quad \sqrt{1-\frac{1}{\sqrt{2}}}=\sqrt{2} \sin(\frac{\pi }{8}),\quad \sqrt{4+2\sqrt{2}}=\csc({\pi\over 8})\]
 and so forth. From these `experimental' results in Table~\ref{table1}, we make the following observations:

\begin{figure}
\[ \includegraphics[scale=0.65]{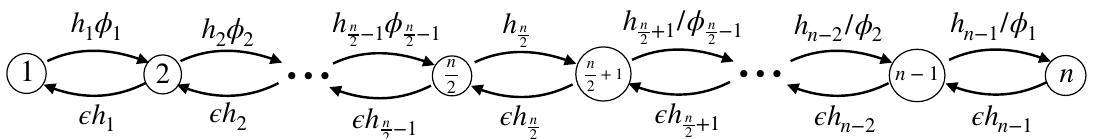}\]
\caption{Metrics on $A_n$ admitting a QLC have at the top square-lengths $h_i \phi_i$ travelling inwards to half way then with the inverse of the $\phi_i$ when travelling outward to the right. At the bottom, the square-lengths are $\eps h_i$  with  $\eps=1$ the physical choice.  We show the case of even $n$; the odd case is similar. \label{metric}}
\end{figure}

\begin{table}
\begin{tabular}{|c|c|c|c|c|c|c|c|}\hline $n$ & $\phi_1$ & $\phi_2$ & $\phi_3$ & $\phi_4$  &  $\phi_5$ & $\phi_6$ &  $\phi_7$ \\
\hline\hline
2\  & 1  \\ \hline
3\  & $2\cos({\pi\over 4})$ &$1\over\phi_1$    \\ \hline
$4^+$ & $2\cos({\pi\over 5})$ &1 & $1\over\phi_1$ \\
$4^-$  & $-2\cos({2\pi\over 5})$ & 1 &   \\ \hline
5\  & $2\cos({\pi\over 6})$ & $\sec({\pi\over 6})$ &$1\over \phi_2$ &$1\over \phi_1$   \\ \hline
$6^+$ & $2 \cos({\pi\over 7})$  & $2 \cos({2\pi\over 7})$&1 & $1\over\phi_2$& $1\over\phi_1$\\
$6^0$ &   $2 \cos({3\pi\over 7})$  & $-2 \cos({\pi\over 7})$ &1 & &  \\
$6^-$ & $-2 \cos({2\pi\over 7})$  & $-2 \cos({3\pi\over 7})$&1 & & \\  \hline
$7^+$ & $2 \cos({\pi\over 8})$ &  $\sqrt{2}\sin({3\pi\over 8})$ & $\csc({3\pi\over 8}) $& $1\over\phi_3$&$1\over\phi_2$ &$1\over\phi_1$ \\
$7^-$ &   $2 \cos({3\pi\over 8})$ & $-\sqrt{2}\sin({\pi\over 8})$& $\csc({\pi\over 8})  $& & &  \\ \hline
8\ (1) & $2 \cos \left(\frac{\pi }{9}\right)$  & $1+2 \cos \left(\frac{4 \pi }{9}\right)$ & $\frac{2 \sin \left(\frac{4\pi }{9}\right)}{\sqrt{3}}$ &1 &$1\over\phi_3$&$1\over\phi_2$ &$1\over\phi_1$  \\
8\ (2) & $-2 \cos \left(\frac{4 \pi }{9}\right)$ &$1+2 \cos \left(\frac{2\pi }{9}\right)$ & $\frac{2 \sin \left(\frac{2 \pi }{9}\right)}{\sqrt{3}}$
 &1 & & & \\
 8\ (3) & $-2 \cos \left(\frac{4 \pi }{9}\right)$ & $1+2 \cos \left(\frac{2\pi }{9}\right)$ & $-\frac{2 \sin \left(\frac{2 \pi }{9}\right)}{\sqrt{3}}$& 1& & &\\
8\ (4) & $-2 \cos \left(\frac{2\pi }{9}\right)$& $1-2 \cos \left(\frac{\pi }{9}\right)$ & $-\frac{2 \sin \left(\frac{ \pi }{9}\right)}{\sqrt{3}}$ &1 & & &  \\ \hline
\end{tabular}
\caption{Table of allowed direction coefficients $\phi_i$ for $n\le 8$\label{table1}}
\end{table}

\begin{remark}

(1) In all cases in the table, we find
\[\phi_{i+1}=\phi_1-{1\over\phi_i},\quad i=1,2,3,4,5,6,7\]
 {\em except for 8(2)} where this holds for $i=1$ but not for $i=2$, for example.

(2) For each $n$ in the table, there is a {\em unique} solution with $\phi_i>0$, shown in the first row. These are reproduced  by  the single formula
\begin{equation}\label{phiphys} \phi_i= {\sin({(i+1) \pi\over n+1})\over \sin({i\pi\over n+1})}={(i+1)_q\over(i)_q};\qquad (i)_q:= {q^{i}-q^{-{i}}\over q^{}-q^{-{1}}};\quad q=e^{\pi \imath\over n+1}\end{equation}
in terms of symmetric $q$-integers. Here $i=1,2,\cdots,n-1$ and the values of $\phi_i$ obey
\[ 2>\phi_1>\phi_2>\phi_3>\cdots > \phi_{\lfloor{n\over 2}\rfloor}\ge 1\]
 with equality in the even case, i.e.  $\phi_i$ decreases from $\phi_1$ at the endpoint towards 1 as we approach the midpoint (and equals $1$  at the midpoint in the even case). After that, the $1/\phi_i$ follows the same pattern going back up to the other endpoint.

 This, along with $\eps=1$ and $h_i>0$, is the {\em unique physical form of the metric} for each $n$ in the sense of positive metric coefficients at each link. One might be able to give an interpretation of negative values as Lorentzian\cite{Ma:sq} but this does not seem reasonable in the present case where all the links are in a line. The metric coefficient or `square-length' is  $h_i\phi_i$ going from $i\to i+1$ and $h_i$ going from $i+1\to i$, and the `direction coefficients' $\phi_i$ is the ratio of these. The above says that there is a longer `square length' travelling into the bulk compared to travelling outward and that this ratio is most at the endpoints and tends to or is 1 in the middle.

(3) In the limit $n\to\infty$, the physical choice in (2) tends to $\phi(i)={i+1\over i}$ as in Figure~\ref{endpoint}. The values of $\phi_i$ for finite $n$ approach these from below and we see that the finite $n$ QRG is a $q$-deformation (for $q$  a root of unity) of the $n\to \infty$ theory.

(4) For each $n$, the unique positive value of $\phi_1$ are roots of a certain polynomial  and $\phi_2$ determined by (1) are roots of a similar polynomial of the same degree. The other allowed values of $\phi_1,\phi_2$ are then all joint solutions of (1) and these two polynomials, modulo the global symmetry mentioned above.

(5) For every allowed quantum metric for  $n\le 8$,  i.e. for every row in the table, there is a {\em unique form of QLC} up to the value of $\tau_1=s$, which is a free parameter required to obey $|s|=1$ for the connection to be $*$-preserving.
\end{remark}

It is expected that the above patterns hold up for all $n$. In particular, it is clear that over $\C$ and with the required `reality' structures, there should be a unique allowed form of quantum metric with positive square-lengths given by free parameters $h_1,\cdots,h_{n-1}>0$ and `direction coefficients' prescribed by (\ref{phiphys}).

\subsection{General solution for the $A_n$ and $\N$}

We now solve for the quantum Riemannian geometry in general $A_n$ motivated by our experience for $n\le 8$, which also serves as a check.  We consider a general metric with coefficients
\[ f_i=\phi_i h_i,\quad f'_i =\eps h_i ,\quad h_i,\phi_i\ne 0,\quad \eps=\pm1    \]
for the metric weights as in (\ref{ghphi}) for increasing and decreasing arrows respectively.

Next, the general form of $\sigma$  is forced by the bimodule map property and  after including the torsion equation, but not yet solving for metric compatibility, to be of the form
\begin{align}
\begin{split}
        \label{eq:8-preservings}
    \sigma(a_i \tens a_{i+1})=\sigma_i a_i\tens a_{i+1},\quad i=1,2,\cdots,n-2,\\
    \sigma(a_i'\tens a'_{i-1})= \sigma'_i a_i'\tens a'_{i-1},\quad i=2,3,\cdots, n-1, \\
\end{split}\\
\begin{split}
    \label{eq:8-preserving1}
    \sigma(a_1\tens a_1')=\tau_1 a_1\tens a_1',    \\
\end{split}\\
\begin{split}
   \label{eq:8-preservingt}
    \sigma(a_i\tens a_i')=\tau_ia_i\tens a_i'+ (1+\tau_i)a_{i-1}'\tens a_{i-1},\quad i=2,\cdots,n-1,\\
    \sigma(a_i'\tens a_i)=\tau_i'a_i'\tens a_i+ (1+\tau_i')a_{i+1}\tens a_{i+1}',\quad i=1,\cdots,n-2,   \\
\end{split}\\
\begin{split}
        \label{eq:8-preservingn}
    \sigma(a'_{n-1}\tens a_{n-1})=\tau_{n-1}'a'_{n-1}\tens a_{n-1},    \\
\end{split}
\end{align}
where the $4(n-2)+2$ parameters are organised into four families with $n-2$ values each for $\sigma_i,\sigma_i'$ and $n-1$ values each for $\tau_i,\tau_i'$ as shown.  The pattern here is that 2-steps in the same line just have one constant as do the back-and-forth steps at the ends where one can only step one way, but when one can go back-and-forth both to the left and to the right, $\sigma$ of one of these is a mixture of both possibilities.

We provide an inductive proof now that there is a QLC for all $n$, limiting attention to metrics of the form (\ref{ghphi}) with $\phi_1,\cdots,\phi_{n-1}$ initially unknown and $\eps=\pm1$ arbitrary, and solving for the connection coefficients. An easier analysis for a 1-way graph line (which is not our case and cannot admit an invertible quantum metric) was in \cite[Exercise~8.1]{BegMa}. In fact, it pays to consider the equations of `$A_\infty$' i.e. the natural numbers $\Bbb N$ as a discrete half-line and obtain any $A_n$ of interest by truncation.

\begin{proposition}\label{propN} For any quantum metric on the natural numbers $\N$ described by $h_i,\eps$, there is a 1-parameter moduli of allowed direction coefficients $\phi_i$ and a 1-parameter family of QLCs as defined iteratively by
\[ \phi_{i+1}=\phi_1-{1\over\phi_i},\quad \tau_{i+1}=- 1 + {\phi_{i+1}\over\phi_i+\eps\tau_i}\]
for arbitrary $\phi_1,\tau_1\ne 0$. The other connection coefficients are then given by
\[ \sigma_{i}={h_{i+1}\phi_{i+1}\over h_{i}\phi_{i}(1+\tau'_{i})},\quad \sigma'_i={h_{i-1}( \phi_{i-1}+\eps\tau_{i-1})\over h_i\phi_i},\quad\tau'_i=\eps{\phi_i-\phi_{i+1}\over\tau_i}.\]
Moreover, for $h_i,\phi_1$ real, the connection is $*$-preserving iff $|\tau_1|=1$.
\end{proposition}
\proof Writing out the equations for metric compatibility, these break down into groups of increasing vertex. The first group is
\begin{align*}
    -f_1 &\tau_1 \tau_1' +  f'_1 = 0 \\
    -f'_1 &\tau_1 \tau_1' +  f_1 - [ f_2 (\tau_2 +1) \sigma_2']= 0
\end{align*}
(this is the same as we saw for $A_2$ except there we do not have the square bracket term because we do not have $f_2,\tau_2,\sigma'_2$). The next group are  6 more equations for the 4 new variables $\sigma_1,\sigma'_2,\tau_2,\tau'_2$ and 2 new parameters $f_2,f_2'$
\begin{align*}
    -f_1& \sigma_1 (\tau_1' +1) +  f_2 = 0\\
    -f'_1&  \sigma_1 (\tau_1' +1) -  f_2  \tau_2 \tau'_2 +  f'_2  = 0\\
    -f'_2&  (\tau_2+1) \sigma'_2 +  f'_1  = 0\\
    -f'_2&  \tau_2  \tau'_2 +  f_2  -  [f_3(\tau_3 + 1)\sigma'_3] = 0\\
    -f'_1&  \tau_1 (\tau'_1 +1) -  f_2  \tau_2  \sigma'_2= 0 \\
    -f'_1&  \sigma_1 \tau'_1 -  f_2  (\tau_2  +1) \tau'_2= 0
\end{align*}
(the equations so far are the  same as we saw for $A_3$ except that there we do not have the square bracket term since there are no variables $f_3,\tau_3,\sigma'_3$). Similarly we have 6 more equations for the 4 new variables $\sigma_2,\sigma'_3,\tau_3,\tau'_3$ and 2 new parameters $f_3,f'_3$
\begin{align*}
  -f_2& \sigma_2 (\tau'_2 +1) +  f_3 = 0\\
    -f'_2&  \sigma_2 (\tau'_2 +1) -  f_3  \tau_3 \tau'_3 +  f'_3  = 0\\
    -f'_3&  (\tau_3+1) \sigma'_3 +  f'_2  = 0\\
    -f'_3&  \tau_3  \tau'_3 +  f_3  -  [f_4(\tau_4 + 1)\sigma'_4] = 0\\
    -f'_2&  \tau_2 (\tau'_2 +1) -  f_3  \tau_3  \sigma'_3= 0 \\
    -f'_2&  \sigma_2 \tau'_2 -  f_3  (\tau_3  +1) \tau'_3= 0
\end{align*}
(the equations so far are the same as we saw for $A_4$ except that there we do not have the square bracket term since no variables $f_4,\tau_4,\sigma'_4$).

The general induction step here is to add a set of six equations for the 4 new variables $\sigma_{i-1},\sigma'_{i},\tau_{i},\tau'_{i}$ and 2 new parameters $f_{i},f'_{i}$,
\begin{align*}
    -f_{i-1}& \sigma_{i-1} (\tau_{i-1}' +1) +  f_i = 0\\
    -f'_{i-1}&  \sigma_{i-1}(\tau_{i-1}' +1) -  f_i  \tau_i \tau'_i +  f'_i  = 0\\
    -f'_i&  (\tau_i+1) \sigma'_i +  f'_{i-1}  = 0\\
    -f'_i&  \tau_i  \tau'_i +  f_i  -  [f_{i+1}(\tau_{i+1} + 1)\sigma_{i+1}'] = 0\\
    -f'_{i-1}&  \tau_{i-1} (\tau'_{i-1} +1) -  f_i  \tau_i  \sigma'_i= 0 \\
    -f'_{i-1}&  \sigma_{i-1} \tau_{i-1}' -  f_i  (\tau_i  +1) \tau'_i= 0
\end{align*}
(which to this point would be the equations for $A_{i+1}$, except for $A_{i+1}$ itself we would not have the square bracket equations due to no $f_{i+1},\tau_{i+1},\sigma'_{i+1}$). This covers the half-line case, and we also noted for later how to extract the $A_n$ solutions from the same analysis.

To solve the system, we rewrite the above $i$th step equations
as
\[ (\tau_i+1)\sigma'_i={f'_{i-1}\over  f'_i},\quad \tau_i\tau'_i={f_i\over  f'_i}- {f_{i+1}\over  f'_{i+1}},\quad\sigma_{i-1}(\tau'_{i-1}+1)={f_i\over f_{i-1}}\]
\[  f_i\sigma'_i-f'_{i-1}\tau_{i-1}=f_{i-1},\quad f'_{i-1}\sigma_{i-1}- f_i\tau'_i=f'_i,\quad {f_i \over f'_i}-{ f_{i+1}\over f'_{i+1}}={f'_i\over f_i }- {f'_{i-1}\over f_{i-1}}\]
If we write $f_i=h_i\phi_i$ and $f'_i= \eps h_i$ for some unknown $\phi_i$ and $\eps = \pm1$  then the last equation is
\[ \phi_{i+1}=\phi_i-{1\over\phi_i}+ {1\over \phi_{i-1}}\]
which, starting off without the $\phi_{i-1}$ term, gives the iteration equation for $\phi_i$ as stated.

Also, from the first and 4th of these $i$th step equations, we get
\[ \tau_{i+1}=-1+{f'_{i}\over f'_{i+1}\sigma'_{i+1}}=-1+{f'_{i}f_{i+1}\over f'_{i+1}(f_i+f'_i\tau_i)}=-1+{\phi_{i+1}\over \phi_i+\eps\tau_i}\]
as stated. Similarly, from the 3rd and the 5th, we get
\[  \tau'_i={f'_{i-1}\sigma_{i-1}-f'_i\over f_i}={f'_{i-1}\over f_{i-1}(1+\tau'_{i-1})}={\eps\over \phi_{i-1}(1+\tau'_{i-1})}-{\eps\over\phi_i}\]
as another (redundant) recursion relation. The 3rd, 4th and 2nd moreover give
\[ \sigma_{i-1}={f_i\over f_{i-1}(1+\tau'_{i-1})},\quad \sigma'_i={f_{i-1}+f'_{i-1}\tau_{i-1}\over f_i},\quad \tau'_i=\eps{\phi_i-\phi_{i+1}\over\tau_i}\]
which can be used to determine $\sigma_{i-1},\sigma'_i$ and obtain $\tau'_i$ from $\tau_i$ as stated (the two sequences are compatible with this relation).

The stated iterative equations have a unique solution given initial values of $\phi_1,\tau_1$. The $\tau_1'$ is determined as
\[ \tau'_1={1\over\tau_1}(\phi_1-\phi_1+{1\over\phi_1})= {1\over\tau_1 \phi_1}.\]

For the last part,  for the connection to be $*$-preserving, we apply $\sigma^{-1}\circ \dagger = \dagger \circ \sigma$ to the relations (\ref{eq:8-preserving1}) and (\ref{eq:8-preservingn}) obtaining the conditions $|\tau_1| = |\tau_{n-1}'| = 1$. Applying the $*$-preserving conditions to (\ref{eq:8-preservings}), we require $\overline{\sigma_i} = 1/\sigma_{i+1}'$ which from their form in the proposition holds if
\[ |\tau_i|^2 = 1 -{ \phi_i\over \phi_{i-1}}.\]
Here, we drop the last term for $i = 1$. This is solved with no conditions beyond $|\tau_1| = 1$, as we prove by induction: if the condition holds for $|\tau_i|^2$ then the recurrence relation for $\tau_{i+1}$ implies that $|\tau_{i+1}|^2 = 1 - \phi_{i+1}/\phi_{i}$ as expected. Similarly, the $*$-preserving conditions for (\ref{eq:8-preservingt}) requires  $|\Delta_i|^2 = 1$, where $\Delta_i = (1+\tau_{i+1})(1+\tau_i') - \tau_{i+1}\tau_i'$. Using the form of $\tau'_i$, this reduces to $|\tau_i|^2 = 1 - \phi_i/\phi_{i-1}$ again.
 \endproof

 \begin{figure}
 \[ \includegraphics[scale=0.67]{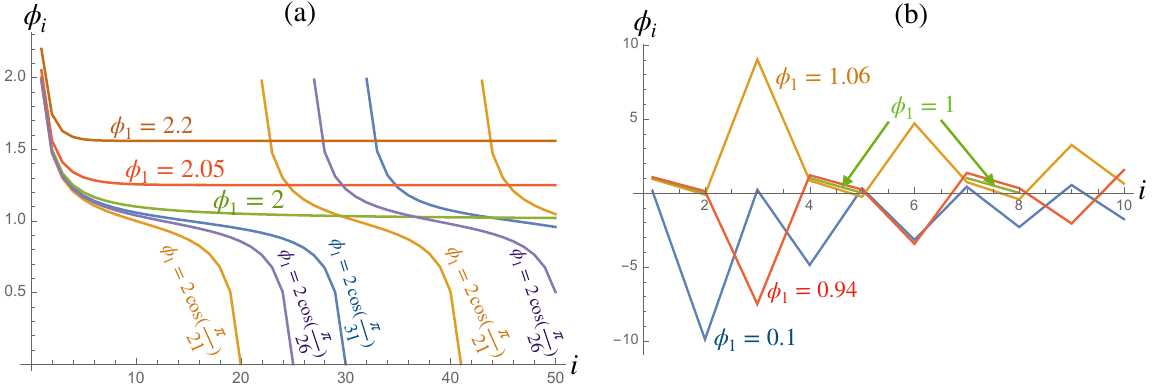}\]
 \caption{(a) $\phi_1\ge 2$ leads to $\phi_i$ asymptotically constant  while $\phi_1<2$ leads to $\phi_i$ oscillatory. Here $\phi_1=2\cos({\pi\over n+1})$ is suitable for $A_n$ and its $\phi_i$ blows up at $i$ a multiple of $n+1$. (b) Smaller $\phi_1$ including $\phi_1=2\cos({\pi\over 3})=1$ and perturbations of it.\label{phiphases}}
 \end{figure}

The iteration for $\phi_i$ here can be done in closed form. If $\phi_1=x$, then
 \[ \phi_i=\frac{1}{2} \left(\sqrt{x^2-4} \left(\frac{1}{\frac{1}{2}-\frac{1}{2} \left(-\sqrt{x^2-4}-x\right)^{-i} \left(\sqrt{x^2-4}-x\right)^i}-1\right)+x\right),\]
 \[ \phi_2=x-{1\over x},\quad \phi_3=\frac{x \left(x^2-2\right)}{x^2-1},\quad \phi_4=\frac{ \left(x^4-3 x^2+1\right)}{x \left(x^2-2\right)},\quad \phi_5=\frac{x \left(x^4-4 x^2+3\right)}{x^4-3 x^2+1},\]
etc.,  and we then use this solution to determine the recursion relations for $\tau_i$. We see from $\phi_3$ that demanding {\em edge-symmetry where $\phi_i=1$ at all $i$, is not an option}. Qualitatively, we see from Figure~\ref{phiphases} that  there are two phases for the system:

\begin{enumerate} \item $\phi_1\ge 2$ (Open phase):  Here $\phi_i$ decays rapidly to an asymptote $\frac{1}{2} \left(\sqrt{\phi_1^2-4}+\phi_1\right)$. This case leads to solutions on the half-line graph $\N$.

 \item $0<\phi_1<2$ (Finite phase): Here $\phi_i$ is oscillatory, could have zeros and singularities and be periodic for certain $\phi_1$ (as illustrated). This case leads to solutions on $A_n$ as a subgraph of $\N$. \end{enumerate}

The critical line started by $\phi_1=2$ between these two regions is particularly simple. It can be approached from either side but more naturally from above.

\begin{proposition}\label{canN} For the half-line graph $\N$, the canonical choice of direction coefficients is given by $\phi_1=2$. If the metric $h_i$ and initial $\tau_1$  are rational then all coefficients of the quantum geometry are rational. In particular, for initial $\tau_1=s=\pm 1$ and $\eps=1$,
\[ \phi_i={i+1\over i},\quad \tau_i={s(-1)^{i-1} \over i};\quad \tau'_i= - \tau_{i+1},\quad  \sigma_i={h_{i+1}\over h_i}(1+\tau_{i+1}),\quad \sigma'_i= {h_{i-1}\over h_i}{1\over 1+\tau_i}.  \]
We refer to this as the {\em canonical quantum Riemannian geometry of  $\N$}.
\end{proposition}
\proof Here $\phi_1=2$ is best approached for the analytic solution from above but one can see directly that the $\phi_i$ stated has this initial value and solves the required recursion relation. For $\tau_i,\tau'_i$, (which are independent of the metric) the recursion has an analytic solution using Pochhammer functions. For example
\[ \tau_1=s,\quad \tau_2=-\frac{16 s+8}{2 (8 s+16)},\quad\tau_3=\frac{120 s+24}{18 (4 s+20)},\quad \tau_4=-\frac{2016 s+864}{288 (12 s+28)} \]
etc. This simplifies greatly when $|s|=1$, namely if $s=e^{\imath\theta}$, then
\[ \tau_i=-\frac{\left(6 i+(-1)^i+3\right) e^{\imath\theta}+2 i+3 (-1)^i+1}{i \left(\left((-1)^i (2 i+1)+3\right) e^{\imath\theta}+3 (-1)^i (2 i+1)+1\right)}.\]
We show the result when $s=\pm 1$ and also for this case the resulting $\tau'_i,\sigma_i,\sigma'_i$. \endproof

We are not claiming this as unique but it it represents by far the simplest solution.  Moreover, although we typically work over $\C$ in mathematical physics, it is striking that this canonical quantum geometry on $\N$ works over the rational numbers $\Bbb Q$. Finally, we turn to the finite case as promised.

\begin{corollary}\label{canAn}  For the finite interval graph  $A_n$ with metric defined by $h_i,\eps,\phi_i$, there is a quantum Riemannian geometry of the form in Proposition~\ref{propN} provided $\phi_1$ is such that the stated iteration leads to $\phi_n=0$ and all preceding $\phi_i\ne 0$. In this case
\[  \phi_{n-i}={1\over \phi_i}.\]
 The physical choice where $\phi_i>0$ for $i=1,\cdots,n-1$ is provided by $\phi_1=2\cos({\pi\over n+1})$ and results in (\ref{phiphys}). E.g., if  $\tau_1=s=\pm1$ and $\eps=1$ then
\[ \tau_i={s(-1)^{i-1}\over (i)_q},\quad  q=e^{\imath \pi\over n+1}\]
as a q-deformation of the solution on $\N$ in Proposition~\ref{canN} with $\tau'_i,\sigma_i,\sigma'_i$ given in terms of this by the same formulae as there. We refer to this as the {\em canonical quantum Riemannian geometry of  $A_n$}
\end{corollary}
\proof We solve the iterative system as before but there is no $f_n$ etc for our truncated graph, so we need
\[  f_{n-1}= f'_{n-1}\tau_{n-1}\tau'_{n-1} \]
in order that the relevant equation in the last group of 6 holds without that square bracketed term. This is
$\tau_{n-1}\tau'_{n-1}=\phi_{n-1}$, which comparing with the general formula $\tau_{n-1}\tau'_{n-1}$ needs  $\phi_{n}=0$. As the preceding $\phi_i$ should all be nonzero, this is the first time this should happen. Next, it follows from the inductive formula for $\phi_i$ in Proposition~\ref{propN} that $\phi_1={1\over \phi_{n-1}}$.  Moreover, assuming $ \phi_i = 1/ \phi_{n-i}$ as induction hypothesis and using the recursive relation for $\phi_{i+1}$ and  $\phi_{n-i}$ gives
\begin{displaymath}
    \phi_{i+1} = \phi_1 - {1\over \phi_i} = \phi_1 - \phi_{n-i}= {1\over \phi_{n-i-1}}
\end{displaymath}
as required, proving the stated assertion.  For $\eps=1$ as here, one can check that (\ref{phiphys}) obeys this has has the required positivity  provided we start with $\phi_1=2\cos({\pi\over n+1})$. 

To find $\tau_i$ for this choice of $\phi_1$, we iterated the recursion relation in Proposition~\ref{propN} to fill out a table of $\tau_i$ values for small $n$ values, see Table~\ref{table2}. Note that standard recursion methods e.g. on Mathematica do not yield a general answer in closed form.  We then `recognised' the general formula as stated. Once found, it is easy enough to check that $\tau_i$  obeys the recursion relation in Proposition~\ref{propN} for $\phi_i=(i+1)_q/(i)_q$ and compute the other values.
 \endproof
 \begin{table}
\begin{tabular}{|c|c|c|c|c|c|c|c|c|}\hline $n$ & $\tau_1$ & $\tau_2$ & $\tau_3$ & $\tau_4$  &  $\tau_5$ & $\tau_6$ &  $\tau_7$&$\tau_8$ \\
\hline\hline
2\  & 1 & -1 \\ \hline
3\  & $1$ &-${1\over\sqrt{2}}$   &1  \\ \hline
4\ & 1& ${1\over 2}(1-\sqrt{5})$ & $-\tau_2$ & -1 \\ \hline
5\  & 1& $-{1\over\sqrt{3}}$ & ${1\over 2}$ &$\tau_2$ & 1   \\ \hline
6\ & 1& $-{1\over 2 \cos({\pi\over 7})}$  &  $(-1)^{3/7}-(-1)^{4/7}$ & $- \tau_3$ & $-\tau_2$& $-1$\\ \hline
7\  & 1&  $-\sqrt{1 - {1\over \sqrt{2}}}$ &  $\sqrt{2}-1$ & $ -{1\over \sqrt{2}}\sqrt{1 - {1\over \sqrt{2}}}$& $\tau_3$&$\tau_2$ &$ 1$ \\ \hline
8\ & 1& $-{1\over 2 \cos \left(\frac{\pi }{9}\right)}$  & $\frac{1}{1+2 \cos \left(\frac{2 \pi }{9}\right)}$ & $-\frac{1}{4\cos \left(\frac{\pi }{9}\right) \cos \left(\frac{2 \pi }{9}\right)}$ &$-\tau_4$ &$-\tau_3$&$-\tau_2$ &$-1$  \\ \hline
\hline
\end{tabular}
\caption{Table of connection coefficients $\tau_i$ for $A_n$ for $n\le 8$,  $s=1$ and the canonical $\phi_1$. We adjoined $\tau_n$ according to the symmetry. \label{table2}}
\end{table}

The canonical choice of $\phi_1$ is illustrated in Figure~\ref{phiphases}. It is also worth noting that for $s=\pm 1$, the $\tau_i$ in this case enjoy symmetries similar to those of $\phi_i$, namely for odd $n$:
  \[ \tau_1=\pm1,\quad\tau_{n-1}=\tau_2,\quad \tau_{n-2}=\tau_3,\quad \tau_{n-1\over 2}=\tau_{n+3\over 2},\quad \tau_{n+1\over 2}=(-1)^{n-1\over 2} \sin({\pi\over n+1})\]
and for even $n$:
 \[ \tau_1=\pm1,\quad\tau_{n-1}=-\tau_2,\quad \tau_{n-2}=-\tau_3,\quad \tau_{{n\over 2}+1}=-\tau_{n\over 2}.\]
Similarly iterating for small $n$ but general $\tau_1=s$ and computing the associated $\tau_i,\tau'_i,\sigma_i,\sigma'_i$  recovers for $A_2,\cdots,A_5$ the explicit values reported in Section~\ref{secsmall} for $\eps=1$.

More generally, without requiring positivity, one can start with $\phi_1=2\cos({j\pi\over n+1})$ for $j=1,2,\cdots,n-1$, but some of these differ only by a sign so the solution they generate can be absorbed in $\eps$, and others can be excluded as some of the $\phi_i$ they generate vanish. Also, for specific $n$ there can be further `irregular' solutions not generated by our method. For example, if we  set $n=8$ then all the solutions for $\phi_1$ such that $\phi_8=0$ are given by
\[ \phi_1:\quad \pm 2\cos({\pi\over 9}),\  \pm 2\cos({2\pi\over 9}),\ \pm2\cos({3\pi\over 9}),\  \pm 2\cos({4\pi\over 9})\]
of which the 3rd solution is just $\pm1$ and can be excluded as not all the $\phi_2,\cdots,\phi_7$ are nonzero. The other three feature in  Table~\ref{table1} (we only listed one of the signs since the other can be absorbed in the choice of $\eps$). This explains the `regular' part of the table but we also see from row 8(2) that for specific $n$, we do not generate all the solutions by our method. Indeed, row 8(2) has the same initial $\phi_1$ as row 8(3) while our method only gives one set of $\phi_i$ for an initial $\phi_1$.

\section{Laplacian and elements of QFT on  $A_3$}\label{seclap}

In the remainder of the paper, we work with only the canonical QRGs on $\N$ and $A_n$ in Proposition~\ref{canN} and Corollary~\ref{canAn},  with $h_i$ real, $s=\pm1$ and $\eps=1$ there. Given this, we will now repurpose $\eps>0$ as a lattice spacing for the continuum limit of the geometry on $\N$. 
In this section, we compute and study the Laplacian $\square\psi=(\ ,\ )\nabla\extd \psi$ for a field $\psi$ on $A_n$ and $\N$, and in Section~\ref{secqg} the Ricci curvature. In both cases, we need the inverse metric, which now looks like 
\[ (a_i,a'_i)={\delta_i\over h_i},\quad (a'_i,a_i)={\delta_{i+1}\over h_i\phi_i}\]
for $i=1,\cdots,n-1$ and is  otherwise zero on basis elements. 

\begin{lemma}\label{lemlap}
    The Laplacian operator $\square f$ of a function $f=\sum_i f(i)\delta_i$ on the vertices of the $A_n$ graph  with $n\ge 3$ has the form
\begin{align*}
    \square f &=  \sum_{i = 1}^{n-1}(f(i+1) - f(i))(,)(\nabla a_i - \nabla a_i')  \\
    &= \left(  (f(1) - f(2)){\tau_1+1\over h_{1}}\right)\delta_1  \\
    &+\sum_{i = 2}^{n-1} \left( (f(i) - f(i-1))(\tau_{i-1}'+1) + (f(i)-f(i+1))(\tau_i+1)   \right)\left({1\over h_{i}} + {1\over h_{i-1}\phi_{i-1}}\right)\delta_i \\
    &+(f(n) - f(n-1))\left({\tau_{n-1}'+1\over h_{n-1}\phi_{n-1}}\right)\delta_n
\end{align*}

\end{lemma}
\proof
Because the calculus is inner, we have  $\extd f  = [\theta,f] = \sum_1^{n-1} (f(i+1) - f(i))(a_i - a_i')$ where $\theta = \sum_1^{n-1}a_i + a_i'$. Using this in the definition of the Laplacian $\square f = (,)\nabla \extd f$ for an arbitrary function $f$ we gives the first expression for the Laplacian.

Next, the general form for the connection corresponding to $a_1$, $a_{n-1}$ and $a_i$ (valid only for $1<i<n-1$) using again that the calculus is inner, is
\begin{align*}
    \nabla a_1  &=  a_1'\tens a_1 - \tau_1 a_1\tens a_1' - \sigma_1 a_1\tens a_2,\\
    \nabla a_1' &=  a_1\tens a_1' + a_2'\tens a_1' - \tau_1' a_1'\tens a_1 -  (\tau_1'+1) a_2\tens a_2',\\
    \nabla a_{n-1}  &= a_{n-2}\tens a_{n-1} + a_{n-1}'\tens a_{n-1} - \tau_{n-1} a_{n-1} \tens a_{n-1}' - (\tau_{n-1}+1) a_{n-2}'\tens a_{n-2},\\
    \nabla a_{n-1}' &=  a_{n-1}\tens a_{n-1}' - \sigma_{n-1}'a_{n-1}' \tens a_{n-2}' - \tau_{n-1}' a_{n-1}'\tens a_{n-1},\\
   \nabla a_i &= a_i'\tens a_i + a_{i-1}\tens a_{i} - \tau_i a_i \tens a_i' - (\tau_i + 1) a_{i-1}'\tens a_{i-1} - \sigma_i a_i \tens a_{i+1}, \\
    \nabla a_i' &= a_i\tens a_i' + a_{i+1}'\tens a_{i}' - \tau_i' a_i' \tens a_i - (\tau_i' + 1) a_{i+1}\tens a_{i+1}' - \sigma_i' a_i' \tens a_{i-1}'.
\end{align*}
Arranging the terms and applying the inverse metric, we recover the explicit form stated.
\endproof
We also use this for the case of $\N$ but without the final values. 

\subsection{$\N$ and its continuum limit}
 The Laplacian for the connection of the Proposition~\ref{canN} with general values of the metric $h_i$ comes out as
 \begin{align*}(\square f)(1)  &=  (f(1) - f(2)){1+s\over h_{1} }\\
 (\square f)(i)&= \left( -(\Delta_\Z f)(i)+ {(-1)^i s\over i}(f(i+1)-f(i-1))\right)\left({1\over h_{i}} + {1\over h_{i-1} (1+{1\over {i-1}})}\right);\quad i>1\end{align*}
 in terms of the usual discrete Laplacian $(\Delta_\Z f)(i)=f(i+1)+f(i-1)-2f(i)$. For the sake of discussion, we now take $s=1$ in order to avoid $(\square f)(1)=0$  for all $f$. An alternative, which amounts to ignoring the $(-1)^i$ term, would be to average the Laplacian between $s=\pm1$.  For reference,
\[(\square_h f)(i)= -(f(i+1)+f(i-1)-2f(i))({1\over h_{i-1}}+{1\over h_i})  \]
was the Laplacian for an infinite line graph $\Z$ with metric weights $h_i$ as found in \cite{ArgMa1}.   Compared to this, we see two effects of the truncation to $\N$, both going as $1/i$ so that they are not visible far from the boundary at $i=1$:
\begin{itemize}
\item[(1)] The metric-dependent factor ${1\over h_{i-1}}+{1\over h_i}$ for $\Z$ {de}creases slightly as $i\to 1$;
\item[(2)] There is a derivative correction but with an alternating sign $(-1)^i$.
\end{itemize}

We now look at both of these in the context of a lattice approximation of $(0,\infty)\subset\R$ sampled at $x=\eps i$, where $i\in \N$ and  $\eps>0$ now denotes the lattice spacing. We consider a function $f$ as either $f(x)$ or $f(i)$ via this correspondence. We implement the lattice spacing by a constant value $h_i=\eps^2$ in the inbound (increasing $i$ direction), but one also has similar results for the more symmetrical $h_i=\eps^2\sqrt{i\over i+1}$. Then the metric-dependent factor is
\begin{equation}\label{betaconsth}  {1\over h_{i}} + {1\over h_{i-1} (1+{1\over i-1})}={1\over \eps^2}\beta^{-1},\quad \beta^{-1}(i)=1+{1\over 1+{1\over i-1}}=2-{1\over i}\end{equation}

(1) We first ignore the term with $(-1)^i$ as this clearly has no continuum limit and we will argue that its effects are minimal. In this case, we have as $\eps\to 0$,
\[\beta(x)={1\over 2}+{\eps \over 4 x}+O(\eps^2),\quad {1\over\eps^2}\Delta_\Z f= {\extd^2f\over\extd x^2}+O(\eps^2)\]
To interpret what happens at order $\eps$ we consider solving the time-independent Schroedinger equation $\square f=4m E f$ for a mass $m$ and energy $E$ in our normalisation of $\square$. We set $\hbar=1$ for present purposes.  Then to $O(\eps^2)$, the equation we are solving is
\[  \left( -{1\over 2 m}{\extd^2\over\extd x^2}-{E\eps\over 2 x}\right)f= Ef.\]
This does not have an immediate parallel with quantum mechanics as the `potential' term is energy-dependent but we can see that energy $E$ shifts by an amount which is $E\eps$ times a `potential' $-{1\over 2 x}$,  with eigenfunctions to $O(\eps^2)$ obtained by solving this.  One could therefore think of this as like a $1\over x^2$ force driving solutions towards the boundary as $x=0$. We illustrate this in
Figure~\ref{figlap}. This shows the solution to
\[ \square f=4  mE f,\quad f(1)=1-\alpha,\ \quad f(2)=1-2\alpha,\quad \alpha={2 m E \eps^2\over 1+2mE\eps^2}\]
where the initial conditions are such that the linearly extrapolated value at the origin is $f(0)=1$ and $4mE f(1)=(\square f)(1)=2(f(1)-f(2))/\eps^2$ as required. The effect of the $(-1)^i$ is to produce ripples in the solution which are less pronounced at larger $x$ and which get faster and smaller as $\eps\to 0$ (since there are more steps  in the range $(0,x)$ for any finite $x$).

\begin{figure}
\[\includegraphics[scale=0.6]{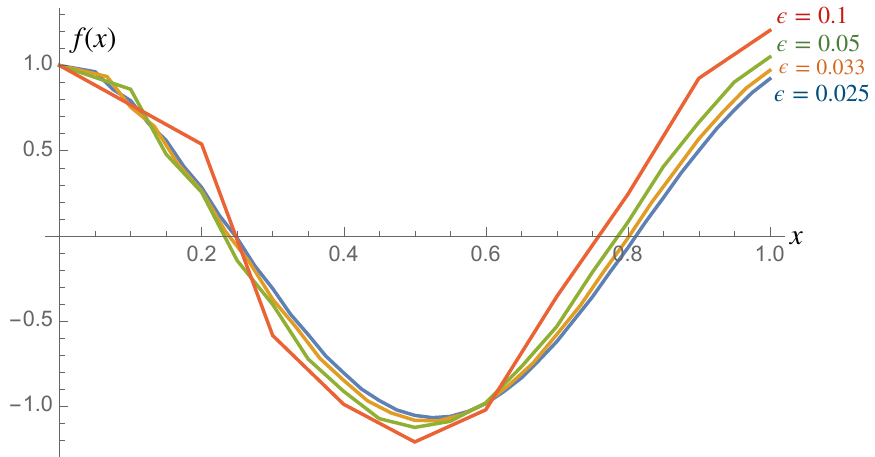}\]
\caption{\label{figlap} Numerical solutions of  $\square f=4mE f$ at $mE=15$ and different values of $h_i=\eps^2$, converging to a smooth solution as $\eps\to 0$.}
\end{figure}

(2) For a theoretical picture of the term with the $(-1)^i$ factor, we discuss two ways to think about this, at least intuitively.  One is to live with the lack of continuity and just keep the $(-1)^i$ factor which in the limit of $\eps\to 0$ stands for an infinitely-rapidly alternating function of $x=i\eps$, but which makes sense for any finite $\eps>0$. In this case, the other parts of the expression have a limit and we obtain
\[\square f= - 2{\extd^2f\over\extd x^2}+ (-1)^i {4\over x}{\extd f\over\extd x}+O(\eps)\]
in so far as this makes sense. The other approach is to sample our functions only at even $i$ and replace $(-1)^i{1\over i}$ by its average value at $i$ and $i+1$, i.e. by ${1\over 2}({1\over i}-{1\over i+1})={1\over 2 i(i+1)}$, which tends to $\eps^2\over 2 x^2$ plus higher order in $\eps$. In this case, one can say, again intuitively, that
\[ \square f=(- {\extd^2f\over\extd x^2}+ {\eps \over x^2}{\extd f\over\extd x})\beta^{-1}(x)+O(\eps^2)=-({\extd^2f\over\extd x^2}+4\beta'{\extd f\over\extd x})\beta^{-1}+O(\eps^2)\]
where we recognise $\beta'(x)=-{\eps\over 4 x^2}$. This leads to a further term ${\eps\over 2 m x^2}f'$ added to the effective `Hamiltonian' in our previous analysis. This contribution no longer has the flavour of  a potential energy but rather of a coupling to an effective background gauge potential. Note that the expression here is not quite  $\beta^{-1}({\extd^2\over\extd x^2} - {\beta'\over 2\beta}{\extd\over\extd x})$, the  classical Laplacian for metric $g=\beta\extd x\tens\extd x$ and connection $\nabla\extd x=-{\beta'\over 2\beta}\extd x\tens\extd x$ in our notations. (Namely, it differs by a factor $-4$ in the $\beta'$ coefficient given that $\beta\approx 1/2$.) We also recall that the limit of the quantum geometry is more precisely a 2-dimensional noncommutative differential calculus rather than a classical calculus.

The overall picture is that the direction-dependent quantum metric on $\N$ cannot be avoided as we approach the $i=1$ boundary and cannot be mapped to an effective continuum limit, but we can begin to get a feel for its physical significance as something like an effective force pushing solutions towards the boundary and possibly a further velocity dependent force. This analysis was for constant metrics $h_i$ or similar. We will mention another natural family of metrics in Section~\ref{secqg}.

\subsection{QFT on a finite lattice interval $A_n$}\label{seclapAn} The Laplacian from Lemma~\ref{lemlap} for the $A_n$ geometry given for the Corollary \ref{canAn}  is
\begin{align*}
    \square f(1) &= (f(1) - f(2))\frac{s+1}{ h_1}, \quad
    \square f(n) = (f(n) - f(n-1)) \frac{1+(-1)^{n} s }{ h_{n-1} } (2)_q,\\
    \square f(i) &= \left( -\Delta_\Z f(i) +  {(-1)^is \over (i)_q}(f(i+1) - f(i-1))  \right)\left(1 + \frac{h_{i}(i-1)_q}{h_{i-1} (i)_q}\right){1\over h_{i}}
\end{align*}
for $i=2,\cdots, n-1$. We used that $(n)_q=1$ and $(n-1)_q=(2)_q$.  It is convenient to  write the Laplacian in the form
\[ \square f=(Lf)\beta^{-1};\quad \beta^{-1}(1)={1\over h_1},\quad\beta^{-1}(n)={(2)_q\over h_{n-1}},\quad \beta^{-1}(i)=\left(1 + \frac{h_{i}(i-1)_q}{h_{i-1}(i)_q}\right){1\over h_{i}},\]
for $i=2,\cdots,n-1$. Then for the free field QFT partition functions etc., we are interested in
\[Z=\int\prod_{i=1}^n\extd \psi(i)\extd\bar\psi(i) e^{{\imath\over \alpha}\sum_{i=1}^n {\mu_i\bar\psi(i)((L\psi)(i)\beta^{-1}(i)}-m^2 \psi(i))}\]
with $\alpha$ a real coupling constant and $\mu_i>0$ a measure of `integration' (now a sum) on $A_n$. This action is quadratic and hence can be evaluated as a determinant. This also applies in the real scalar field case where we have $\prod_i\extd\psi(i)$. 

As an example, we let $n=3$ and $s=1$. We have $(2)_q=\sqrt{2}$ and $(3)_q=1$ and weights $\mu_i$, we have an action for a complex (or real) 3-vector $\psi=(\psi_1,\psi_2,\psi_3)$. In these terms, the action is quadratic,
\begin{align*} S[\psi]&=\bar\psi_1 \mu _1 \left(\frac{2 \left(\psi _1-\psi _2\right)}{h_1}-m^2 \psi _1\right)+\bar\psi_3 \mu _3 \left(-m^2 \psi _3\right)\nonumber\\
&\quad  +\bar\psi_2 \mu _2 \left(\left(\frac{1}{h_2}+\frac{1}{\sqrt{2}h_1 }\right) \left(\left(-\frac{1}{\sqrt{2}}-1\right) \psi _1+\left(\frac{1}{\sqrt{2}}-1\right) \psi _3+2 \psi _2\right)-m^2 \psi _2\right)\\
&= \bar\psi.B.\psi\end{align*}
for a matrix
\[B=\begin{pmatrix}
 \mu _1 \left(\frac{2}{h_1}-m^2\right) & -\frac{ \mu _12}{h_1} & 0 \\
- \mu_2\left(1+\frac{1}{\sqrt{2}}\right) \left(\frac{1}{h_2}+\frac{1}{\sqrt{2}h_1}\right)\quad & \mu_2 \left({2}\left(\frac{1}{h_2}+\frac{1}{\sqrt{2}h_1}\right)- m^2\right)\quad & \mu _2 \left(\frac{1}{\sqrt{2}}-1\right) \left(\frac{1}{h_2}+\frac{1}{\sqrt{2}h_1}\right) \\
 0 & {0} & \mu _3 \left(-m^2\right) \end{pmatrix}\]
 now including the mass term. Hence, the partition function $Z$ is again given as usual for a Gaussian via the determinant
\begin{align*}\det(B)&=  \frac{\mu _1 \mu _2 \mu _3 m^2}{h_1^2 h_2^2} \Big(h_1^2  h_2 m^2\left(-h_2 m^2+2  \right)+\left(\sqrt{2}+2\right) h_1 h_2^2 m^2 \\
 &\qquad\qquad\qquad\quad + \left(\sqrt{2}-2\right) h_1h_2  -h_2^2\left(\sqrt{2}-1\right) \Big). \end{align*}
 For the constant case $h_1=h_2$, this simplifies to
 \[ \det(B)=-\frac{\mu _1 \mu _2 \mu _3m^2}{h_1^2}   \left( h_1 m^2 \left(h_1 m^2-\sqrt{2}-4\right)+1\right).\]
 One can similarly compute correlation functions.

\section{Curvatures and elements of quantum gravity on $n=3$}\label{secqg}

In this section, we compute the curvatures in terms of the $h_i$ real parameters and $s=\pm 1$ and $\eps=1$, i.e. for the canonical QRGs on $\N$ in Proposition~\ref{canN} and on $A_n$  by Corollary~\ref{canAn} with $\phi_1=2\cos({\pi\over n+1})$.  We then use the lifting map (\ref{imin}) to define the Ricci tensor as in \cite{BegMa} and the Ricci scalar by $S=(\ ,\ )({\rm Ricci})$ where $(\ ,\ )$ is the inverse metric. Curvature tensors will depend on the $h_i$ only through the ratio
\[ \rho_i={h_{i+1}\over h_i},\]
which we use throughout the section. We also repurpose $\eps$ as the lattice spacing in the case of $\N$.

In general, the Riemann curvature (\ref{curv}) for our form of connection reduces to
\begin{align*}
    \text{R}_\nabla a_1 &= 0,\\
    \text{R}_\nabla a_1' &= (\tau_1'(\sigma_{2}'-\tau_1) +\sigma_{2}' )a_1'\wedge a_1\tens a_1' + (  \tau_1'(\tau_2' - \sigma_1) + \tau_2' )a_1'\wedge a_1\tens a_{2},\\
    \text{R}_\nabla a_i &= (\tau_i(\sigma_{i-1}-\tau_i') + \sigma_{i-1} - \sigma_i(\tau_{i+1}+1))a_i\wedge a_i'\tens a_i + ( \tau_i(\tau_{i-1}-\sigma_i') \\
    &\quad + \tau_{i-1} ) a_i\wedge a_i'\tens a_{i-1}',\\
    \text{R}_\nabla a_i' &= (\tau_i'(\sigma_{i+1}'-\tau_i) + \sigma_{i+1}' - \sigma_i'(\tau_{i-1}'+1) )a_i'\wedge a_i\tens a_i' + (  \tau_i'(\tau_{i+1}' - \sigma_i) \\
    &\quad + \tau_{i+1}' )a_i'\wedge a_i\tens a_{i+1},\\
    \text{R}_\nabla a_{n-1} &= ( \tau_{n-1}(\sigma_{n-2}-\tau_{n-1}') + \sigma_{n-2} )a_{n-1}\wedge a_{n-1}'\tens a_{n-1}\\
    &\quad + (\tau_{n-1}(\tau_{n-2} - \sigma_{n-1}') + \tau_{n-2})a_{n-1}\wedge a_{n-1}'\tens a_{n-2}',\\
    \text{R}_\nabla a_{n-1}' &= 0
\end{align*}
for  $i= 2,\cdots, n-2$ on $A_n$ and the same without the final cases on $\N$.

\subsection{Curvatures for $\N$}  The results for the canonical solution in Proposition~\ref{canN}  with $s=\pm1$ are as follows. For the Riemann curvature, we find
\begin{align*}
    \text{R}_\nabla a_1 &= 0, \\
    \text{R}_\nabla a_1' &= \left( \rho_1\frac{ (1-2s) }{4}-\frac{(1+2s)}{6} \right)a_1'\wedge a_1 \tens a_2 - \left( \frac{ (s+2)}{\rho_1 (s-2)}+\frac{1}{2} \right)a_1'\wedge a_1 \tens a_1' , \\
    \text{R}_\nabla a_i &= \left( \rho_{i-1}\frac{ \left(i-(-1)^i s\right)^2}{i^2 }{-}\frac{\rho_i\left(i+1+(-1)^i s\right)^2}{(i+1)^2}-\frac{1}{i(i+1)} \right)a_i\wedge a_i' \tens a_i\\
    &\quad + (-1)^i s\left( \frac{ \left(i-(-1)^i s\right)}{(i-1)i}+\frac{1}{ \rho_{i-1} \left(i-(-1)^i s\right)} \right)a_i\wedge a_i' \tens a_{i-1}' , \\
    \text{R}_\nabla a_i' &= \left( -\frac{ \left(i+(-1)^i s\right)}{ \rho_{i-1}\left(i-(-1)^i s\right)}+\frac{\left(i+1-(-1)^i s\right)}{\rho_i \left(i+1+(-1)^i s\right)}-\frac{1}{i(i+1)}  \right)a_i'\wedge a_i \tens a_i' \\
    &\quad + {(-1)^i s\over (i+1)}\left( \frac{   \left(i+1-(-1)^i s\right)}{(i+2 ) }+\rho_{i}\frac{  \left(i+1+(-1)^i s\right)}{(i+1 )} \right)a_i'\wedge a_i \tens a_{i+1} 
\end{align*}
for $i\ge 2$. The Ricci tensor for the lift (\ref{imin}) is then
\begin{align*}
    \text{Ricci} &= \left( \rho_1\frac{ (s-2) s}{4}-\frac{s (s+2)}{6} \right)  a_1\tens a_2 -\left( \frac{ (s+2)}{\rho_1 (s-2)}+\frac{1}{2} \right)  a_1\tens a_1' \\
    &+{{1\over 2}}\sum_{i\ge 2}\Big\{\phi_i\left( -\frac{ \left(i+(-1)^i s\right)}{\rho_{i-1} \left(i-(-1)^i s\right)}+\frac{ \left(i+1-(-1)^i s\right)}{\rho_i \left(i+1+(-1)^i s\right)}-\frac{1}{i(i+1)} \right)  a_i\tens a_i' \\
       &\qquad+{(-1)^i s\phi_i\over (i+1)}\left( \frac{\left(i+1-(-1)^i s\right)}{ (i+2)}+\rho_{i}\frac{ \left(i+1+(-1)^i s\right)}{(i+1)  } \right)  a_i\tens a_{i+1}  \\
        &\qquad+{1\over\phi_i}\left( \rho_{i-1}\frac{ \left(i-(-1)^i s\right)^2}{i^2 }{-}\frac{\rho_i\left(i+1+(-1)^i s\right)^2}{(i+1)^2}-\frac{1}{i(i+1)} \right) a_i'\tens a_i \\
    &\qquad+{(-1)^i s\over\phi_i}\left( \frac{  \left(i-(-1)^i s\right)}{(i-1) i }+\frac{1}{ \rho_{i-1} \left(i-(-1)^i s\right)} \right)  a_i'\tens a_{i-1}' \Big\}.
\end{align*}
There are no Ricci flat solutions but note that the only coefficients that do not decay $O({1\over i})$ for generic $
\rho_i$  are the coefficients of $a_i\tens a'_i$ {and $a_i'\tens a_i$, which asymptote to ${1\over\rho_i}-{1\over\rho_{i-1}}$ and $\rho_{i-1}-\rho_i$ respectively}. 

Contracting with $(\ ,\ )$, the Ricci scalar is then
\begin{align*}
    \text{S}(1) &= {-} \frac{1}{2 h_1} \left(1+\frac{2 (s+2)}{\rho_1 (s-2)}\right), \\
    \text{S}(2) &= {-} \frac{1}{8 h_2} \left(1-\frac{6  (s+2)}{\rho_1 (s-2)}+\frac{6 (s-3)}{\rho_2 (s+3)}\right),  \\
    \text{S}(i) &=  {-}{1\over 2h_i}\left( {1\over i^2}{+}\rho_{i-1}\frac{ (i-1) }{i^3}{-}\rho_{i-1}\rho_{i-2}\frac{\left(i-1+(-1)^i s\right)^2}{i^2 } {+}\rho^2_{i-1}{(i-1)^2 \left(i-(-1)^i s\right)^2\over i^4}    \right. \\
    &
\quad\qquad \qquad\left.  +\frac{(i+1) \left(i+(-1)^i s\right)}{\rho_{i-1}i  \left(i-(-1)^i s\right)}  -\frac{(i+1) \left(i+1-(-1)^i s\right)}{\rho_i i \left(i+1+(-1)^i s\right)} \right), 
\end{align*}
with
\[ S(i)=  {-}{1\over 2 h_i}\left(\rho_{i-1}(\rho_{i-1}-\rho_{i-2})+ {1\over \rho_{i-1}}-{1\over\rho_i}\right)+ O({1\over i})\]
for generic $h_i$. In the constant $h_i$ case, however, we have to look to the next order and then one finds
\begin{align*}S(i)&=-{(1+4(-1)^i s) \over  h_1 i^2}+ O({1\over i^3})\end{align*}
 which has a non-continuum alternating term suppressed for large $i$, in line with such a term term in the Laplacian in Section~\ref{seclap}. 

 Alternatively, we can land exactly on $S=0$, in fact on any prescribed function for the curvature, provided we use an oscillatory $h_i$ which will then not have a classical limit itself. We explore this option next.

\begin{proposition}\label{flatmetric} On $\N$, there is a unique metric $\{h_i\}$ up to normalisation such that $S=0$, given by
\[s=1:\quad h_{i}^{flat} =2 h_1 {(2\lfloor{i\over 2}\rfloor+1)^2\over (i+1)}=2 h_1\begin{cases}  i+1 &  i\ \text{even}\\ \frac{i^2}{(i+1)}  & i\ \text{odd}\end{cases},\]
\[ s=-1:\quad h_i^{flat}=2 h_1{\lceil{i\over 2}\rceil^2\over (i+1)}={ h_1\over 2}\begin{cases} { i^2\over (i+1)} &  i\ \text{even}\\ i+1  & i\ \text{odd}\end{cases},\]
for any inital value $h_1$.
\end{proposition}
\proof From the form of $S(i)$, it is clear that we can solve iteratively to find $h_i$ for any initial $h_1$. Doing this for $s=\pm1$ gives the solutions shown. \endproof

For the rest of the section, we focus on $s=1$ but there is a similar story for $s=-1$.  First note that
 setting $s=1$ and $h_1=\eps^3$ for a small number $\eps>0$ and repeating the analysis in Section~\ref{seclap}, the metric-dependent factor in the Laplacian becomes
\[ {1\over h_{i}}+ {1\over h_{i-1}(1+{1\over i-1})}={1\over\eps^3}\beta^{-1}(i);\quad \beta^{-1}(i)={1\over i}+O({1\over i^2})\]
in place of (\ref{betaconsth}). We can take the continuum limit with the leading order $\beta^{-1}(i)={1\over i}= {\eps \over x}$ and do the parallel analysis to (1) in Section~\ref{seclap}. Ignoring the $(-1)^i$ differential term as we did before, gives that $\square f=4m E f$ becomes the Airy equation
\begin{equation}\label{airy} {\extd^2 f\over \extd x^2}+4 m E x f=0\end{equation}
with a real decaying cosine-wave-like solution if $4mE>0$ and, say, $f(0)=1, f'(0)=0$. In addition, we can expect ripples in the discrete solution visible for small $i$ due to the even values of $\beta^{-1}(i)$ and due to the $(-1)^i$ differential term as discussed in Section~\ref{seclap}. Meanwhile, the QFT action depends on the measure $\mu_i$ and if we take the obvious choice $\mu_i=h_i$ then this cancels the $1/x$ in the continuum limit and we obtain a multiple of the free field action (again ignoring the suppressed $(-1)^i$ term in the Laplacian), which is perhaps reasonable as the curvature is zero.

Next, we consider metrics near to the above flat one in a conformal sense,
\[ h_i=h_i^{flat}g_i;\quad \rho_i=\rho_i^{flat} \eta_i;\quad \eta_i={g_{i+1}\over g_i}\]
with $h^{flat}$ from Proposition~\ref{flatmetric} for $s=1$. Then a calculation with $h_1=\eps^3$ and $i=x/\eps$ as above and working to leading order in $\eps$, gives
\[  S(x)=-{1\over 2 h_i^{flat}g_i} \left(  \eta_{i-1}(\eta_{i-1}-\eta_{i-2})+ {1\over \eta_{i-1}}-{1\over\eta_i}  \right)=-{1 \over 4 \eps  x g }({1\over \eta^2}{+}\eta) {\extd \eta\over \extd x},\]
where
\[ \eta_i=1+{g_{i+1}-g_i\over g_i}=1+\eps g^{-1}{\extd g\over \extd x}+ O(\eps^2).\]
Putting this in, we have to leading oder in $\eps$,
\[ S(x)=-{ g^{-1} \over 4 x  }   {\extd \over\extd x}(g^{-1}{\extd g\over\extd x})=-{e^{-\psi} \over 4  x   } {\extd^2\psi \over\extd x ^2}\]if we write $g(x)=e^{\psi(x)}$ for a real scalar field $\psi$.

We briefly consider the Einstein-Hilbert action for such metric fluctuations near the scalar-flat metric, expressed in $\psi(x)$. We need to fix the measure $\mu_i$ in
\[ S[h]=\sum_i \mu_i S(i)\]
and based on experience in \cite{Ma:haw} for $\Z$, we take $\mu_i=h_i=h_i^{flat}g_i$. The theory behind how to choose this measure is not clearly understood, but we expect some power of the metric. Classically, one would have $\sqrt{\det(g)}$ for the measure but in \cite{Ma:haw} it gave more reasonable answers not to take a square root, related to $\Omega^1$ there being 2-dimensional. Our $\Omega^1$ is not exactly a free module but is more like this far from the boundary. In this case,
\[ S[\psi]=\sum_{i=1}^\infty h_i ^{flat}g_iS(i)\to -{\eps\over 2 }\int_0^\infty \extd x {\extd \over\extd x}({\extd\psi\over \extd x})= { \eps\over 2}({\extd \psi\over\extd x})(0^+)\]
to leading order as $\eps\to 0$, given that $h_i^{flat}=2\eps^2 x$ to leading order and assuming our fields decay at $\infty$. The $\eps$  can be absorbed in $\mu$ or in a coupling constant in front of the action. The action here is topological and appears to amount to a trivial theory on the boundary at $x=0^+$ (approaching from the bulk), but could be more interesting before we take the continuum limit and if we look more closely at the boundary.

\subsection{Curvatures for $A_n$} We proceed from the general expression for the curvature and put in the connection in Corollary~\ref{canAn}. However, the $\tau_i,\tau'_i,\sigma_i,\sigma_i'$ are exactly a $q$-deformation of the formulae for $\N$ in the sense that all integers $i-1,i,i+1,i+2$ are replaced by $(i-1)_q,(i)_q,(i+1)_q,(i+2)_q$ respectively.  With that change, the formulae for $R_\nabla$ for are exactly as before except that now $R_\nabla a_{n-1}'=0$ and $R_\nabla a_{n-1}$ drops the term with $\rho_{n-1}$.  One can also simplify with $(n)_q=1$ and $(n-1)_q=(2)_q$. Likewise $R_\nabla a_1=0$ and $R_\nabla a'_1$ drops the term with $1/\rho_0$. Thus,
\[ R_\nabla a_1=0,\quad R_\nabla a_1'= - \left(\frac{s+(2)_q}{\rho_1 (s-(2)_q)}+\frac{1}{(2)_q}\right)a_1'\wedge a_1\tens a_1'  ,\]
\[  R_\nabla a_{n-1}=\left(\frac{\rho_{n-2} \left((2)_q+(-1)^n s\right)^2}{(2)_q^2}-\frac{1}{(2)_q }\right)a_{n-1}\wedge a_{n-1}'\tens a_{n-1},\quad R_\nabla a'_{n-1}=0. \]
For the  Ricci tensor the sum over $i$ is $q$-deformed and truncated, but for the $i=1$ term we drop the $1/\rho_0$ in the coefficient of $a_1\tens a'_1$ and the $a_1'\tens $ terms altogether, and for the $i=n-1$ terms we drop the  $\rho_{n-1}$ in the coefficient of $ a_{n-1}'\tens a_{n-1}$ and the $a_{n-1}\tens$ terms altogether.  Thus,
\begin{align*} {\rm Ricci}&=  {s\over 2}\left( \rho_1{s-(2)_q\over (2)_q}-{(2)_q+s\over (3)_q}\right)a_1\tens a_2- {1\over 2}\left({(2)_q((2)_q+s)\over \rho_1(s-(2)_q)}+1\right)a_1\tens a_1' \\
&\quad +{1\over 2} \sum_{i=2}^{n-2}(q-{\rm def\ previous})+{1\over 2}\left(\rho_{n-2}{((2)_q+(-1)^ns)^2\over (2)_q}-1\right)a_{n-1}'\tens a_{n-1}\\
&\quad -{(-1)^ns\over 2}\left({(2)_q+(-1)^n s\over (3)_q}+{(2)_q\over \rho_{n-2}((2)_q+(-1)^n s)}\right) a'_{n-1}\tens a'_{n-2}
\end{align*}

There are no Ricci flat solutions. The Ricci scalar is then
\begin{align*}
    \text{S}(1) &= -\frac{1}{2 h_1}\left(1+ {(2)_q(s+(2)_q)\over \rho_1(s-(2)_q)}\right),\\
    \text{S}(2) &= -\frac{(3)_q}{2 h_2 (2)_q}\left( {1\over (2)_q(3)_q}- {s+(2)_q\over \rho_1 (s-(2)_q)}+ {s-(3)_q\over \rho_2(s+(3)_q)}\right),\\
    \text{S}(i) &= -\frac{(i+1)_q}{2h_{i} (i)_q}\left({1\over (i)_q(i+1)_q}+ { ((i)_q+(-1)^is)\over \rho_{i-1}((i)_q-(-1)^is)}- {((i+1)-(-1)^i s)\over \rho_i((i+1)_q+(-1)^i s)}   \right) \\
    &-{(i-1)_q^2\over 2 h_{i}(i)_q^2}\rho_{i-1}\left({1\over (i-1)_q(i)_q}- \rho_{i-2}{((i-1)_q+(-1)^is)^2\over (i-1)_q^2} +\rho_{i-1}{((i)_q-(-1)^i s)^2\over (i)_q^2}   \right), \\
    \text{S}(n-1)&=- \frac{(3)_q^2}{2 h_{n-2} (2)_q^2}\left({1\over (2)_q (3)_q} +\rho_{n-2}{((2)_q+(-1)^n s)^2\over (2)_q} - \rho_{n-3}{((3)_q-(-1)^n s)^2\over (3)_q}  \right),\\
    \text{S}(n) &= - \frac{1}{2h_{n-1}}\left( (2)_q-  \rho_{n-2}((2)_q+(-1)^n s)^2 \right)
\end{align*}
for $n\ge 3$. For $n=3$ we have $S(2)=S(n-1)=0$. These formulae show how the geometry of $A_n$ $q$-deforms that of $\N$. As with Proposition~\ref{flatmetric}, there is again a unique metric $h^{flat}$ up to overall scale such that $S=0$.

We conclude with a small example for $n= 3, s = 1$. Then $(2)_q=\sqrt{2}$, $(3)_q=1$ and we obtain
\[S=\left({(3+2\sqrt{2})\over\sqrt{2} h_2}-{1\over 2 h_1}\right)\{1,0,-(3-2\sqrt{2})\}\]
at the three points. This vanishes at $h_2=(4+3\sqrt{2})h_1$. Next, if we write $\mu_1=h_1, \mu_3=h_2$ then  get for the Einstein-Hilbert action
\[ S[\rho]:=\sum_i\mu_iS(i)={(3+2\sqrt{2})\over\sqrt{2}\rho}+{(3-2\sqrt{2})\over 2}\rho -{1\over \sqrt{2}}-{1\over 2};\quad \rho= {h_2\over h_1}\]
but note that we can get any coefficients for the two powers of $\rho$ by scaling $\mu_i$. Sticking with the obvious values, if we ignore the constant then
\[ Z=\int \extd h_1\extd h_2 e^{-{1\over G}({c\over\rho}+\rho)}=\int_0^\infty h_1\extd h_1\int_0^\infty\extd \rho e^{-{1\over G}({c\over\rho}+\rho)};\quad c=24+17\sqrt{2}\]
for a real positive coupling constant $G$. The first integral is an infinite volume which we ignore, while the $\rho$ integrals converge for the calculation of expectation values,
\[ \int_0^\infty\extd\rho e^{-{1\over G}({c\over\rho}+\rho)}\rho^m= 2 c^{\frac{m+1}{2}} K_{m+1}\left(\frac{2 \sqrt{c}}{G}\right)\] as BesselK functions. These diverge as $G\to 0$ and as $G\to \infty$, but the expectation  behave like
\[ \<\rho^m\>\to \begin{cases} c^{m\over 2} & G\to 0\\ \infty & G\to \infty\end{cases},\]
while, for example, the relative uncertainty increases from 0 to a limit
\[ {\Delta\rho\over \<\rho\>}:={ \sqrt{\<\rho^2\>-\<\rho\>^2}\over \<\rho\>} \to 1;\quad {\<\rho^2\>\over\<\rho\>^2}\to 2 \]
as $G\to \infty$ (the `strong gravity' limit). This looks quite reasonable for a theory of quantum gravity on 3 points in the sense that it follows the same pattern as other models\cite{Ma:sq,ArgMa1,LirMa}. 

\section{Concluding remarks}

We have found that the quantum geometry of a finite discrete lattice in 1 dimensions is intrinsically $q$-deformed, but from solving for a quantum Riemannian geometry and not from assuming a quantum group. Also note that the structure of $\Omega$  is rather nontrivial as $\Omega^1,
\Omega^2$ are not free modules, i.e. the differential structure is not parallelisable with a global basis. In fact, the exterior algebra $\Omega$ is a quotient of the preprojective algebra of Dynkin type $A_n$, which is intimately connected to the representation theory and structure of the Lie algebra $sl_{n+1}$. This suggests that there could be a  reduced quantum group $u_{q}(sl_{n+1})$ (in some conventions) at play with $q$ an even $2(n+1)$-th root of unity. This should be explored further as well as the  implied links to the physics of spin-chains as quantum integrable systems. Finally, the mathematical structure suggests that it should also be interesting to look at quantum Riemannian geometry (and quantum spin `chains', reduced quantum groups) for the other types of Dynkin graphs.

On the physics side, our emergence of $q$-deformation out of discretisation together with the belief that $q$-deformation in 2+1 quantum gravity with point sources corresponds to a cosmological constant (see \cite{MaSch} for an overview) suggests that switching on the cosmological constant could be equivalent to discretisation of quantum geometry. The models are very different, but if we naively equate $q=e^{\imath\pi\over n+1}$ with $q=e^{-\imath {\lambda_p\over\lambda_c}}$ where $\lambda_c$ is the cosmological constant length scale and $\lambda_p$ the Planck length, then we get $n\sim \pi \lambda_c/\lambda_p$. Here,  $\lambda_c=1/\sqrt{\Lambda}$ for positive cosmological constant $\Lambda$. If a similar ideas were to hold for our observed 3+1 universe, where it appears that $\lambda_c\sim 10^{26}m$, this would give $n\sim 3\times 10 ^{61}$. This is merely for comparison and we not suggesting any particular role for the QRG of the $A_n$ graph with this many nodes in 2+1 quantum gravity.  Morally speaking, however, the idea that we discretise spacetime as a way to regularise quantum gravity, seems to match up with switching on the cosmological constant and could explain why the latter is very small compared to the Planck scale and yet nonzero. This would agree with other evidence\cite{MaTao} that it could arise as a consequence of quantum gravity effects that render spacetime noncommutative. Looking more broadly at a possible role of $A_n$ in quantum gravity, we speculate that this could perhaps be relevant to the geometry of an open string of $n$ Planck lengths. 

There is also a lot more to be done on the QRG of $\N$. This we found to be rational,  i.e. if the metric coefficients are rational then so are all Christoffel symbols etc. We found that there were non-continuum $(-1)^i$ terms which have no continuum limit but which would be pushed to the origin at $x=0$ in the limit. In that case the unique (up to scale) flat metric on $\N$ led for the Laplacian to the Airy equation in the bulk, but the behaviour around $x=0$ needs much more attention. This would also affect our conclusion for the field theory of a conformal factor $e^\psi$ on this flat metric, which we found to be topological and hence reducing to the value at the origin, but again on the assumption of a particular choice $\mu_i=h_i$ of the measure for `integration' prior to taking the continuum limit. There is as of yet no general theory for this measure as well as a lack of a variational calculus in general. Also, we only examined the limit of scalars but the limit of the Ricci tensor etc. can also be studied provided we limit the differential structure correctly. For $\Z$, this is 2-dimensional and limits to a certain noncommutative 2-dimensional calculus on the line\cite{ArgMa1} with the classical calculus as a quotient. This is the reason why we have any curvature in the first place. For $\N$, the calculus is more complicated particularly around the origin, but in the bulk we would expect a similar limit. It also remains to look at particle creation and other possible quantum gravity effects\cite{MukWin}, such as found for $\Z$ in \cite{Ma:haw} but now adapted to $\N$. It could also be of interest to use both $A_n,\N$ as parts of higher-dimensional models as in \cite{ArgMa2}, in place $\Z_n$.

Finally, we took a first look at quantum theory  and quantum gravity in a functional integral approach on $A_n$ for $n=3$, with the quantum gravity expectation values and relative uncertainties in line with other models \cite{Ma:sq,ArgMa1,LirMa}. Because of the $q$-deformation, these models also have a rich structure in the numerics which, however, manages to stay real for the particular root of unity in the picture.  These are some directions for further work.

\section*{Acknowledgements} The first author was partially supported by CONACyT (M\'exico)

\section*{Author Declarations}

The authors have no conflicts to disclose. Data sharing is not applicable to this article as no new data were created nor analyzed in this study.

\end{document}